\documentclass[11.5pt,twoside, a4paper, english, reqno]{amsart} %%,draft
\usepackage{amsaddr}%%[foot]
\usepackage{a4wide}
\usepackage{amscd}
\usepackage{amssymb}
\usepackage{amsthm}
\usepackage{graphicx}
\usepackage{amsmath,calrsfs}
\usepackage{latexsym}
\usepackage{enumitem,dsfont}

\usepackage{upref}
\usepackage[colorlinks]{hyperref}%backref
\usepackage{color,graphics}
\usepackage{upref,hyperref,color}

\usepackage[usenames,dvipsnames]{xcolor}
\usepackage{pdfsync}
\usepackage{ulem,cancel,pgf}

\usepackage{frcursive}
\theoremstyle{plain}
\newtheorem{theorem}{Theorem}[section]
\theoremstyle{plain}
\newtheorem{lemma}[theorem]{Lemma}
\newtheorem{proposition}[theorem]{Proposition}
\newtheorem{corollary}[theorem]{Corollary}

\theoremstyle{definition}
\newtheorem{definition}{Definition}[section]

\newtheorem{remark}{Remark}[section]
\newtheorem{claim}{Claim}[section]

\newtheorem*{maintheorem*}{Main Theorem}
\newtheorem*{maincorollary*}{Main Corollary}

\DeclareFontFamily{U}{BOONDOX-calo}{\skewchar\font=50 }
\DeclareFontShape{U}{BOONDOX-calo}{m}{n}{
	<-> s*[1.05] BOONDOX-r-calo}{}
\DeclareFontShape{U}{BOONDOX-calo}{b}{n}{
	<-> s*[1.05] BOONDOX-b-calo}{}
\DeclareMathAlphabet{\mathcalb}{U}{BOONDOX-calo}{m}{n}
\SetMathAlphabet{\mathcalb}{bold}{U}{BOONDOX-calo}{b}{n}
\DeclareMathAlphabet{\mathbcalb}{U}{BOONDOX-calo}{b}{n}
\numberwithin{equation}{section} \allowdisplaybreaks

\title[$ W^{s, G}_{\text{loc}}(\Omega)- $Comparison principle]{Comparison principle for Singular Fractional $ g- $Laplacian Problems}
\date{\today}
\author[Abdelhamid Gouasmia]{Abdelhamid Gouasmia}
\address[Abdelhamid Gouasmia]{
Department of Mathematics, Faculty of Sciences And Technology,\\
Mohamed Cherif Messaadia University,\\
P.O.Box 1553, Souk Ahras 41000, Algeria.\\[4pt]
Laboratoire d'equations aux d\'{e}riv\'{e}es partielles non lin\'{e}aires et histoire des math\'{e}matiques,\\
Ecole Normale Sup\'{e}rieure,\\
B.P. 92, Vieux Kouba, 16050 Algiers, Algeria.}
\email[Abdelhamid Gouasmia]{abdelhamid.gouasmia@univ-soukahras.dz}
\author[Kaushik Bal]{Kaushik Bal}
\address[Kaushik Bal]{
	Indian Institute of Technology Kanpur, Kanpur, India}
\email[Kaushik Bal]{kaushik@iitk.ac.in}
\usepackage{comment}
\begin{document}
%	\rmdefault%for serif font family,
%%	\sfdefault %for sans serif font family
	%\ttdefault
	
\begin{abstract}
	In this paper, we establish a novel comparison principle of independent interest and prove the uniqueness of weak solutions within the local Orlicz--Sobolev space framework, for the following class of fractional elliptic problems:
	\begin{equation*}
	(-\Delta)^{s}_{g} u = f(x) u^{-\alpha} + k(x) u^{\beta}, \quad u > 0 \quad \text{in } \Omega; \quad u = 0 \quad \text{in } \mathbb{R}^{N} \setminus \Omega,
	\end{equation*}
	where \( \Omega \subset \mathbb{R}^{N} \) is a smooth bounded domain, \( \alpha > 0 \), and \( \beta > 0 \) satisfies a suitable upper bound. Here, \( (-\Delta)^{s}_{g} \) denotes the fractional \( g \)-Laplacian, with \( g \) being the derivative of a Young function \( G \). The function \( f \) is assumed to be nontrivial, while \( k \) is a positive function, and both \( f \) and \( k \) are assumed to lie in suitable Orlicz spaces. Our analysis relies on a refined variational approach that incorporates a \( G \)-fractional version of the D\'iaz--Saa inequality together with a \( G \)-fractional analogue of Picone's identity. These tools, which are of independent interest, also play a key role in the study of simplicity of eigenvalues, Sturmian-type comparison results, Hardy-type inequalities, and related topics.
\end{abstract}

\maketitle
%\begin{footnotesize}
\hspace*{0.45cm}\textbf{MSC:} 35J75, 35R11, 35J62.  \\

\setcounter{tocdepth}{1}
%\end{footnotesize}
%\tableofcontents
\section{Introduction}
Since the 1930s, Orlicz and Orlicz-Sobolev spaces have been extensively studied, particularly due to their numerous applications in physics and engineering, where problems involving \textit{nonstandard growth} differential equations arise. These spaces generalize the classical Lebesgue and Sobolev spaces by replacing the standard power function \( t^{p} \) with a more general convex function \( G(t) \). For a comprehensive discussion and various applications, we refer to \cite{adams2003sobolev, bahrouni2023espaces, bonder2019fractional, krasnoselskii1961convex, lieberman1991natural, martinez2008minimum} and the references therein. In recent years, considerable attention has been devoted to the study of nonlocal elliptic operators, particularly in connection with problems exhibiting anomalous diffusion and transport phenomena. In this context, a fractional counterpart of the classical Orlicz-Sobolev space naturally emerges. In particular, the authors in \cite{bonder2019fractional} introduced the fractional Orlicz-Sobolev space associated with an \( N \)-function \( G \) and a fractional parameter \( 0 < s < 1 \), defined as
{\small \[
W^{s, G}(\Omega) = \left\lbrace u \in L^{G}(\Omega) :\quad \int_{\Omega} \int_{\Omega} G\left(\dfrac{u(x) - u(y)}{\left| x - y\right| ^{s}}\right) \dfrac{dx dy}{\left| x - y\right| ^{N}} < \infty \right\rbrace,
\]}

\noindent
where \( \Omega \) is an open subset of \( \mathbb{R}^{N} \), and \( L^{G}(\Omega) \) denotes the Orlicz space (see Section \ref{sc2} for precise definitions). In the same work, a generalization of the fractional \( p \)-Laplacian, referred to as the fractional \( g \)-Laplace operator, is introduced as  
{\small \[
(- \Delta)_{g}^{s} u(x) = \textbf{P.V.} \int_{\mathbb{R}^{N}} g\left( \dfrac{\left| u(x) - u(y) \right|}{\left| x - y \right|^{s}} \right) 
\dfrac{u(x) - u(y)}{\left| u(x) - u(y) \right|} \cdot \dfrac{dy}{\left| x - y \right|^{N + s}},
\]}

\noindent
where \textbf{P.V.} denotes the Cauchy principal value, and \( g = G' \) is the derivative of the \( N \)-function \( G \).

\noindent
Recently, considerable attention has been devoted to elliptic problems involving the fractional \( g \)-Laplacian, which exhibits nonlocal behavior. Several fundamental aspects have been explored, including structural properties such as existence, the maximum principle, and interior Sobolev and Lipschitz regularity \cite{alberico2021fractional, bahrouni2020basic, bonder2022interior, ochoa2024existence}. Various functional inequalities, such as Hardy and Poincaré inequalities, have also been investigated (see \cite{bal2022hardy}). Additionally, a range of partial differential equation (PDE) problems, including eigenvalue problems, has been analyzed \cite{bahrouni2023variational, salort2020eigenvalues}.  Despite these advancements, relatively few results are available concerning singular problems involving the fractional \( g \)-Laplacian. To the best of our knowledge, the only existing studies in the literature addressing this topic are \cite{bal2024singular} and \cite{boujemaa2023new}, with the latter focusing on a new class of fractional Orlicz-Sobolev spaces with variable order.  Let us bring our attention to the results in \cite{bal2024singular}, where the authors establish the existence and regularity of \( W^{s, G}_{\text{loc}}(\Omega) \)-solutions for the following singular problem with a variable exponent:
\begin{equation}\label{P2}
(-\Delta)^{s}_{g} u  = h(x) u^{-q(x)}, \quad u > 0 \quad \text{in } \Omega; \quad u = 0, \quad \text{in } \mathbb{R}^{N} \setminus \Omega,
\end{equation}
where \( \Omega \subset \mathbb{R}^N \) is a smooth bounded domain. The assumptions include \( 0 < s < 1 \), a non-negative function \( q \in C^{1}(\overline{\Omega}) \), and \( h \) belonging to a suitable Orlicz space. Moreover, the functions \( G \) and \( g \) satisfy the conditions:
\begin{itemize}
	\item[\textbf{(H1)}] The function \( g = G' \) is absolutely continuous and thus differentiable almost everywhere.
\item[\textbf{(H2)}] The function \( g' \) is nondecreasing on \( (0, \infty) \), and consequently on \( \mathbb{R} \setminus \{0\} \).
	\item[\textbf{(H3)}] The function \( G \) satisfies
	\[
	\int_{0}^{1}\frac{G^{-1}(\tau)}{\tau^{\frac{N + s}{N}}} d\tau < \infty \quad \text{and} \quad \int_{1}^{\infty}\frac{G^{-1}(\tau)}{\tau^{\frac{N + s}{N}}} d\tau  = \infty.
	\]
	\item[\textbf{(H4)}] There exist constants \( p^{+} \) and \( p^{-} \) such that
	\[
	1 < p^{-} - 1 \leq \dfrac{t g'(t)}{g(t)} \leq p^{+} - 1 \leq \infty, \quad t > 0.
	\]
\end{itemize}
It is noteworthy that a substantial body of research has been dedicated to investigating various classes of singular problems through diverse analytical techniques. In particular, these studies encompass both the \( p \)-Laplace operator, defined for \( p > 1 \) as  
\[
\Delta_{p} u = \text{div} \left( \left|\nabla u \right|^{p-2} \nabla u \right),
\]
and the fractional \( q \)-Laplacian, which, for a fixed parameter \( s \in (0,1) \), is given by  
\begin{align*}
( -\Delta ) ^{s}_{q} u(x) &:= 2\, \textbf{P.V.} \int_{\mathbb{R}^N} \frac{|u(x) - u(y) |^{q-2} ( u(x) - u(y)) }{| x-y|^{N + s q} }\, dy,
\end{align*}
where \( q > 1 \).  This operator represents a particular case of the fractional \( g \)-Laplacian when \( G(t) = t^{q} \). Moreover, numerous studies have examined equations involving mixed-type operators that integrate both local and nonlocal structures. For a comprehensive overview of recent developments concerning existence and nonexistence results, as well as other qualitative properties, we refer to \cite{barrios2015semilinear, canino2017moving, durastanticomparison, giacomoni2012singular} and the references therein.  One of the primary challenges in establishing uniqueness results for singular problems stems from the fact that, in general, solutions do not belong to the finite energy space, as highlighted in the aforementioned results for \eqref{P2}. To address this issue, several approaches have been proposed. In particular, the variational technique introduced in \cite{canino2004variational} was employed to establish uniqueness results in the context of the local semilinear case. This method was later refined in \cite{canino2016uniqueness} and has since been extensively applied to prove uniqueness results for problems involving local, nonlocal, and mixed operators; see, for instance, \cite{arora2021regularity, bal2024regularity, dhanya2024interior, giacomoni2021sobolev, gouasmia2024uniqueness}, though this list is not exhaustive.  
This naturally gives rise to the following question:
\[
\textbf{Question:} \quad \text{Can analogous uniqueness results be established for problem \eqref{P2}?}
\]
To explore this question, we consider a more general framework that implicitly encompasses problem \eqref{P2}. Specifically, we examine the following problem:
\begin{equation}\label{GP}\tag{P}
(-\Delta)^{s}_{g} u = f(x) u^{-\alpha} + k(x) u^{\beta}, \quad u > 0 \quad \text{in } \Omega; \quad u = 0 \quad \text{in } \mathbb{R}^{N} \setminus \Omega,
\end{equation}
where the parameter \(\beta\) satisfies \( 0 < \beta < p^{-} - 1 \) and \( \alpha > 0 \). The function \( f \) is nonzero and belongs to an appropriate Orlicz space, while \( k \) is a positive function in a suitable Orlicz space. The primary objective of this paper is to establish the uniqueness of infinite energy solutions to problem \eqref{GP}, provided such solutions exist. Our investigation relies on a new comparison principle developed in this work under specific conditions (see Theorem \ref{Theorem1} below). This approach refines the method introduced in~\cite{canino2016uniqueness} {(see also~\cite{carvalho2018type} for related techniques), by employing a weak comparison principle derived from the convexity properties of the involved functionals to effectively handle the difficulties introduced by the nonlocal terms. Moreover, appropriate test functions are carefully chosen in conjunction with the condition on \( \beta \) previously described.

\noindent
On the other hand, we investigate the uniqueness of finite-energy solutions and explore the possibility that such solutions may exhibit unbounded behavior under fairly general conditions, as described by the following problem:
\begin{equation}\label{GGP}\tag{GP}
(-\Delta)^{s}_{g} u = F(x, u), \quad u > 0 \quad \text{in } \Omega; \quad u = 0 \quad \text{in } \mathbb{R}^{N} \setminus \Omega,
\end{equation}
where the nonlinearity \( F \) satisfies the following assumptions:
\begin{enumerate}
	\item[\textbf{(F1)}] \( F : \Omega \times \mathbb{R}^{+} \to \mathbb{R}^{+} \) is a Carathéodory function, that is, \( F(x, \cdot) \) is continuous on \( \mathbb{R}^{+} \) for almost every \( x \in \Omega \), and \( F(\cdot, s) \) is measurable for all \( s > 0 \).
	\item[\textbf{(F2)}] For almost every \( x \in \Omega \), the map \( s \mapsto \dfrac{F(x, s)}{s^{p^{-}-1}} \) is non-increasing on \( \mathbb{R}^{+} \setminus \{0\} \).
\end{enumerate}
The structure of the paper is as follows: Section~\ref{sc2} presents the necessary preliminary results in a precise manner. We establish a D\'{\i}az and Saa type inequality for the \( g \)-Laplacian, which leads to a uniqueness result. These results are of independent interest and may be applied to other classes of problems. The statements of the main theorems are also provided in this section. Section~\ref{sc3} is devoted to the proofs of these results.

\section{Definitions and Main Results}\label{sc2}

This section is organized into two subsections. The first subsection provides a concise overview of the fundamental concepts and key properties related to $N$-functions and Orlicz–Sobolev spaces (see \cite{bahrouni2020basic, bonder2019fractional, fernandez2022holder, krasnoselskii1961convex}), along with supplementary definitions and auxiliary results that are crucial for the development of this paper. The second subsection is devoted to the presentation of the main results.
\subsection{Preliminaries and Functional Setting}

Let $\Omega \subset \mathbb{R}^N$ be a bounded domain with $C^{0,1}$-regular boundary. We adopt the following definitions:

\medskip
\noindent
A function $G : \mathbb{R}^{+} \to \mathbb{R}^{+}$ is called an \textit{$N$-function} if it is continuous, convex, and satisfies:
\[
G(t) = 0 \iff t = 0, \quad \frac{G(t)}{t} \to 0 \text{ as } t \to 0, \quad \frac{G(t)}{t} \to \infty \text{ as } t \to \infty.
\]
Equivalently, the function $G$ can be represented as
\[
G(t) = \int_0^t g(s) \, ds,
\]
where $g : \mathbb{R}^{+} \to \mathbb{R}^{+}$ is a non-decreasing and right-continuous function satisfying
\[
g(t) = 0 \iff t = 0, \quad \text{and} \quad \lim_{t \to \infty} g(t) = \infty.
\]
In addition, throughout this work, we assume that \( g \in C^{1}(\mathbb{R}^{+}) \). The following examples of the function \( G \) fall within our framework:	
	\begin{itemize}
		\item[(1)] \( G(t) := t \). In this case, the associated operator reduces to the standard fractional Laplacian: 
		\[
		(-\Delta)^{s}_{g} u = (-\Delta)^{s} u.
		\]		
		
		\item[(2)] \( G(t) := \frac{1}{p} t^{p} \), where \( p \geq 2 \). This corresponds to the fractional \( p \)-Laplacian operator.		
		
		\item[(3)] \( G(t) := \frac{1}{p} t^{p} \left( \left| \ln(t) \right| + 1 \right) \), where \( p \geq 2 \). This function exhibits slower growth near zero and is useful in modeling singular behaviors.		
		
		\item[(4)] \( G(t) := \frac{1}{p} t^{p} + \frac{1}{q} t^{q} \), with \( p, q \geq 2 \). In this case, the operator takes the form
		\[
		(-\Delta)^{s}_{g} u = (-\Delta)^{s}_{p} u + (-\Delta)^{s}_{q} u,
		\]
		which corresponds to the \((p,q)\)-fractional Laplacian.
	\end{itemize}

\noindent It is established in \cite[Theorem 4.1]{krasnoselskii1961convex} that the upper bound in condition \textbf{(H4)} is equivalent to the so-called $\Delta_2$-condition (also known as the doubling condition), which states that
\begin{equation}\label{equ1}\tag{$\Delta_2$}
G(2t) \leq 2^{p^{+}} G(t) \quad \text{for all } t \geq 0.
\end{equation}
Moreover, the same condition \textbf{(H4)} implies the following inequality:
\begin{equation}\label{equ2}
2 < p^{-} \leq \frac{t\, g(t)}{G(t)} \leq p^{+} \leq \infty, \quad \text{for all } t > 0.
\end{equation}

\noindent
We say that $G$ grows essentially faster than another $N$-function $H$, denoted $H \prec\prec G$, if for every $k > 0$,
\[
\lim_{t \to \infty} \frac{H(k t)}{G(t)} = 0.
\]

\noindent
The Sobolev conjugate of $G$, denoted by $G_*$, is defined as
\[
G_{*}^{-1}(t) := \int_0^t \frac{G^{-1}(\tau)}{\tau^{\frac{N + s}{N}}} \, d\tau.
\]

\noindent
The complementary function of $G$, denoted by $\overline{G}$, is defined for all $t \geq 0$ by
\[
\overline{G}(t) := \sup_{\tau \geq 0} \left( t \tau - G(\tau) \right).
\]
This leads to the following version of Young’s inequality:
\begin{equation}\label{equ3}
ab \leq G(a) + \overline{G}(b) \quad \text{for all } a, b \geq 0.
\end{equation}

\noindent
Let \( u : \mathbb{R}^{N} \to \mathbb{R} \) be a measurable function. We define the associated modular functionals as
\begin{equation*}
\Phi_{G}(u) := \int_{\Omega} G\left( \left| u(x) \right| \right) \, dx, \quad  
\Phi_{s, G}(u) := \int_{\Omega} \int_{\Omega} G\left( \frac{\left| u(x) - u(y) \right|}{\left| x - y \right|^{s}} \right) \frac{dx \, dy}{\left| x - y \right|^{N}}.
\end{equation*}

\begin{lemma}[\cite{bal2024singular}, Lemma 2.1]\label{lemma1}
	Let \( G \) be an \( N \)-function satisfying \eqref{equ2}. Then, for all \( \lambda \geq 0 \) and all \( t > 0 \),  the following inequalities hold:
	\[
	\min\left\{ \lambda^{p^{-}}, \lambda^{p^{+}} \right\} G(t) \leq G(\lambda t) \leq \max\left\{ \lambda^{p^{-}}, \lambda^{p^{+}} \right\} G(t).
	\]
\end{lemma}

\noindent
The norm in the Orlicz space \( L^{G}(\Omega) \) is defined by
\[
\left\| u \right\|_{L^{G}(\Omega)} := \inf\left\lbrace \lambda > 0 \mid \Phi_{G}\left( \frac{u}{\lambda} \right) \leq 1 \right\rbrace.
\]

\noindent
A direct consequence of Young’s inequality~\eqref{equ3} yields the following lemma:
\begin{lemma}[Hölder inequality]\label{lemma3}
	Let \( G \) be an \( N \)-function, and let $\overline{G}$ be its complementary function. Then, for every $ u \in L^{G}(\Omega) $ and $ v \in L^{\overline{G}}(\Omega) $, the following inequality holds:
	\[
	\int_{\Omega} uv\, dx \leq 2 \left\| u \right\|_{L^{G}(\Omega)} \left\| v \right\|_{L^{\overline{G}}(\Omega)}.
	\]
\end{lemma}

\noindent
The  fractional  Orlicz--Sobolev space \( W^{s, G}(\Omega) \) is defined by
\[
W^{s, G}(\Omega) := \left\lbrace u \in L^{G}(\Omega) \,\middle|\, \exists\, \lambda > 0 \text{ such that } \Phi_{s, G}\left( \frac{u}{\lambda} \right) < \infty \right\rbrace,
\]
and is endowed with the norm
\[
\left\| u \right\|_{W^{s, G}(\Omega)} := \left\| u \right\|_{L^{G}(\Omega)} + \left[ u \right]_{s, G},
\]
where
\[
\left[ u \right]_{s, G} := \inf\left\lbrace \lambda > 0 \,\middle|\, \Phi_{s, G}\left( \frac{u}{\lambda} \right) \leq 1 \right\rbrace,
\]
is called the $(s, G)$-Gagliardo seminorm.  We also define the subspace \( W^{s, G}_{0}(\Omega) \) by
\[
W^{s, G}_{0}(\Omega) := \left\lbrace u \in W^{s, G}(\mathbb{R}^{N}) \,\middle|\, u \equiv 0 \text{ in } \mathbb{R}^{N} \setminus \Omega \right\rbrace,
\]
which is endowed with the norm
\[
\| u \|_{W^{s, G}_{0}(\Omega)} := \left[ u \right]_{s, G}.
\]

\noindent
As a consequence of Lemma~\ref{lemma1}, we derive estimates relating the modulars to their norms.
\begin{lemma}\label{lemma2}
	Let \( G \) be an \( N \)-function satisfying condition \eqref{equ2}. Then, for every \( u \in L^{G}(\Omega) \), the following inequality holds:
	\[
	\min\left\{ \left\| u \right\|_{L^{G}(\Omega)}^{p^{-}},\, \left\| u \right\|_{L^{G}(\Omega)}^{p^{+}} \right\} 
	\leq \Phi_{G}(u) 
	\leq 
	\max\left\{ \left\| u \right\|_{L^{G}(\Omega)}^{p^{-}},\, \left\| u \right\|_{L^{G}(\Omega)}^{p^{+}} \right\}.
	\]
	Moreover, for every \( u \in W^{s, G}(\Omega) \), we have
	\[
	\min\left\{ \left\| u \right\|_{W^{s,G}(\Omega)}^{p^{-}},\, \left\| u \right\|_{W^{s,G}(\Omega)}^{p^{+}} \right\} 
	\leq \Phi_{s, G}(u) 
	\leq 
	\max\left\{ \left\| u \right\|_{W^{s,G}(\Omega)}^{p^{-}},\, \left\| u \right\|_{W^{s,G}(\Omega)}^{p^{+}} \right\}.
	\]
\end{lemma}

\begin{lemma}[Poincaré inequality {\cite[Theorem 2.12]{fernandez2022holder}}]\label{lemma4}
	Let \( G \) be an \( N \)-function. Then there exists a constant \( C = C(N, s, p^{+}, p^{-}, \Omega) > 0 \) such that, for every \( u \in W^{s, G}_{0}(\Omega) \), the following inequality holds:
	\[
	\int_{\mathbb{R}^{N}} G\left( \left| u(x) \right| \right) \, dx \leq C \int_{\mathbb{R}^{N}} \int_{\mathbb{R}^{N}} G\left( \frac{\left| u(x) - u(y) \right|}{\left| x - y \right|^{s}} \right) \frac{dx \, dy}{\left| x - y \right|^{N}}.
	\]
\end{lemma}

\begin{remark}\label{remark1}
	By the Poincaré-type inequality, the norm \( \| \cdot \|_{W^{s, G}_{0}(\Omega)} \) is equivalent to the norm \( \| \cdot \|_{W^{s, G}(\mathbb{R}^{N})} \) on the space \( W^{s, G}_{0}(\Omega) \). Moreover, if \( G \) is an \( N \)-function satisfying condition \eqref{equ2} and assumption \textbf{(H3)}, then, as shown in \cite[Theorem 1.2]{bahrouni2019embedding}, the following continuous embedding holds:
	\[
	W^{s, G}_{0}(\Omega) \hookrightarrow L^{G_{*}}(\Omega),
	\]
	together with the compact embedding
	\[
	W^{s, G}_{0}(\Omega) \hookrightarrow L^{H}(\Omega), \quad \text{for every } H \prec\prec G_{*}.
	\]
\end{remark}

\noindent
We now recall a fundamental result that will play a key role in the subsequent analysis.

\begin{proposition}[\( G \)-fractional hidden convexity]\label{proposition1}
	Let \( G \) be an \( N \)-function satisfying assumption \textbf{(H4)}, and suppose that \( 1 < q \leq p^{-} \). For every pair of nonnegative functions \( u_0, u_1 \), define
	\[
	\sigma_t(x) := \left[ (1 - t)\, u_0^{q}(x) + t\, u_1^{q}(x) \right]^{\frac{1}{q}}, \quad
	t \in [0, 1],\; x \in \mathbb{R}^N.
	\]
	Then, for all \( t \in [0, 1] \) and all \( x, y \in \mathbb{R}^N \), the following inequality holds:
	\[
	G\left( \left| \sigma_t(x) - \sigma_t(y) \right| \right) \leq (1 - t)\, G\left( \left| u_0(x) - u_0(y) \right| \right) + t\, G\left( \left| u_1(x) - u_1(y) \right| \right).
	\]
\end{proposition}

\begin{proof}
	As established in \cite[Lemma 4.1]{franzina2014fractional}, we have
	\[
	\left| \sigma_t(x) - \sigma_t(y) \right| \leq \left[ (1 - t)\, \left| u_0(x) - u_0(y) \right|^{q} + t\, \left| u_1(x) - u_1(y) \right|^{q} \right]^{\frac{1}{q}}, \quad \text{for all } t \in [0, 1] \text{ and } x, y \in \mathbb{R}^N.
	\]
We first observe that the mapping \( s \mapsto G\left(s^{1/q}\right) \) is convex. Indeed, a straightforward computation, combined with the regularity assumption \( g \in C^{1}(\mathbb{R}^{+}) \) and condition \textbf{(H4)}, yields
	\[
	\dfrac{d^{2}}{d s^{2}} G\left(s^{1/q}\right) = \dfrac{s^{\frac{1}{q} - 2}}{q^{2}}\, g\left(s^{1/q}\right) \left[ 1 - q + \dfrac{s^{1/q}\, g'\left(s^{1/q}\right)}{g\left(s^{1/q}\right)} \right] \geq 0.
	\]
	Hence, for all \( a, b \geq 0 \) and \( t \in [0, 1] \), the convexity of \( s \mapsto G\left(s^{1/q}\right) \) implies that
	\[
	G\left( \left[ (1 - t) a + t b \right]^{1/q} \right) \leq (1 - t)\, G\left( a^{1/q} \right) + t\, G\left( b^{1/q} \right).
	\]
	Applying this inequality with \( a = |u_0(x) - u_0(y)|^{q} \) and \( b = |u_1(x) - u_1(y)|^{q} \) concludes the proof.
\end{proof}
\begin{proposition}[\( G \)-fractional Picone inequality]\label{proposition2}
	Let \( G \) be an \( N \)-function, and let \( u, v \) be measurable functions satisfying \( u > 0 \), \( v \geq 0 \), and \( 1 < q \leq p^{-} \). Then there exists a constant \( C = C(p^{-}, p^{+}, G(1)) > 0 \) such that the following inequality holds:
	\begin{equation}\label{equ9}
	\begin{aligned}
	&g\left( \left| u(x) - u(y) \right| \right) \dfrac{u(x) - u(y)}{\left| u(x) - u(y) \right|} 
	\left[ \dfrac{v(x)^{q}}{u(x)^{q-1}} - \dfrac{v(y)^{q}}{u(y)^{q-1}} \right] \\
	&\leq C \max\left\lbrace 
	\left( \dfrac{G\left( \left| v(x) - v(y) \right| \right)}{G(1)} \right)^{\frac{q}{p^{+}}}, 
	\left( \dfrac{G\left( \left| v(x) - v(y) \right| \right)}{G(1)} \right)^{\frac{q}{p^{-}}}
	\right\rbrace \\
	&\quad \times \max\left\lbrace 
	\left( \dfrac{G\left( \left| u(x) - u(y) \right| \right)}{G(1)} \right)^{\frac{p^{-} - q}{p^{+}}}, 
	\left( \dfrac{G\left( \left| u(x) - u(y) \right| \right)}{G(1)} \right)^{\frac{p^{+} - q}{p^{-}}}
	\right\rbrace.
	\end{aligned}
	\end{equation}
\end{proposition}

\begin{proof}
	We begin by noting that if \( u(x) = u(y) \), then inequality \eqref{equ9} holds trivially. Thus, we may assume \( u(x) \neq u(y) \). Without loss of generality, suppose \( u(x) > u(y) \), since inequality \eqref{equ9} is invariant under the permutation \( (x, y) \mapsto (y, x) \). Moreover, if 
	\[
	\dfrac{v(x)^{q}}{u(x)^{q-1}} - \dfrac{v(y)^{q}}{u(y)^{q-1}} \leq 0,
	\]
	then inequality \eqref{equ9} is again trivially satisfied. Otherwise, we distinguish between the following four cases:
	
	\medskip
	\noindent
	\textbf{Case 1:} \( |u(x) - u(y)| < 1 \) and \( |v(x) - v(y)| < 1 \). In this case, using \eqref{equ2} and the fractional Picone inequality \cite[Proposition 4.2]{brasco2014convexity}, we obtain
	\begin{equation*}
	\begin{aligned}
	&g\left( |u(x) - u(y)| \right) 
	\dfrac{u(x) - u(y)}{|u(x) - u(y)|} 
	\left[ \dfrac{v(x)^{q}}{u(x)^{q-1}} - \dfrac{v(y)^{q}}{u(y)^{q-1}} \right] 
	\\[4pt]
	&\leq 
	p^{+} G(1) |u(x) - u(y)|^{p^{-} - 2} (u(x) - u(y)) 
	\left[ \dfrac{v(x)^{q}}{u(x)^{q-1}} - \dfrac{v(y)^{q}}{u(y)^{q-1}} \right] 
	\\[4pt]
	&\leq 
	p^{+} G(1)  |v(x) - v(y)|^{q} \, |u(x) - u(y)|^{p^{-} - q}
	\\[4pt]
	&\leq 
	p^{+} G(1)^{1 + \frac{p^{-}}{p^{+}}} 
	\left( G(|v(x) - v(y)|) \right)^{\frac{q}{p^{+}}} 
	\left( G(|u(x) - u(y)|) \right)^{\frac{p^{-} - q}{p^{+}}}.
	\end{aligned}
	\end{equation*}
	
	\medskip
	\noindent
	\textbf{Case 2:} \( |u(x) - u(y)| > 1 \) and \( |v(x) - v(y)| > 1 \). Proceeding similarly, we obtain
	\begin{equation*}
	\begin{aligned}
	&g\left( |u(x) - u(y)| \right) 
	\dfrac{u(x) - u(y)}{|u(x) - u(y)|} 
	\left[ \dfrac{v(x)^{q}}{u(x)^{q-1}} - \dfrac{v(y)^{q}}{u(y)^{q-1}} \right] 
	\\[4pt]
	&\leq 
	p^{+} G(1) |u(x) - u(y)|^{p^{+} - 2} (u(x) - u(y)) 
	\left[ \dfrac{v(x)^{q}}{u(x)^{q-1}} - \dfrac{v(y)^{q}}{u(y)^{q-1}} \right] 
	\\[4pt]
	&\leq 
	p^{+} G(1)  |v(x) - v(y)|^{q} \, |u(x) - u(y)|^{p^{+} - q}
	\\[4pt]
	&\leq 
	p^{+} G(1)^{1 + \frac{p^{+}}{p^{-}}} 
	\left( G(|v(x) - v(y)|) \right)^{\frac{q}{p^{-}}} 
	\left( G(|u(x) - u(y)|) \right)^{\frac{p^{+} - q}{p^{-}}}.
	\end{aligned}
	\end{equation*}
	
	\medskip
	\noindent
	\textbf{Case 3:} \( |u(x) - u(y)| > 1 \) and \( |v(x) - v(y)| < 1 \). By the same arguments,
	\begin{equation*}
	\begin{aligned}
	&g\left( |u(x) - u(y)| \right) 
	\dfrac{u(x) - u(y)}{|u(x) - u(y)|} 
	\left[ \dfrac{v(x)^{q}}{u(x)^{q-1}} - \dfrac{v(y)^{q}}{u(y)^{q-1}} \right] 
	\\[4pt]
	&\leq 
	p^{+} G(1) |u(x) - u(y)|^{p^{+} - 2} (u(x) - u(y)) 
	\left[ \dfrac{v(x)^{q}}{u(x)^{q-1}} - \dfrac{v(y)^{q}}{u(y)^{q-1}} \right] 
	\\[4pt]
	&\leq 
	p^{+} G(1)  |v(x) - v(y)|^{q} \, |u(x) - u(y)|^{p^{+} - q}
	\\[4pt]
	&\leq 
	p^{+} G(1)^{\frac{(p^{+})^{2} -q(p^{+} - p^{-})}{p^{+} p^{-}}} 
	\left( G(|v(x) - v(y)|) \right)^{\frac{q}{p^{+}}} 
	\left( G(|u(x) - u(y)|) \right)^{\frac{p^{+} - q}{p^{-}}}.
	\end{aligned}
	\end{equation*}
	
	\medskip
	\noindent
	\textbf{Case 4:} \( |u(x) - u(y)| < 1 \) and \( |v(x) - v(y)| > 1 \). Again, we deduce
	\begin{equation*}
	\begin{aligned}
	&g\left( |u(x) - u(y)| \right) 
	\dfrac{u(x) - u(y)}{|u(x) - u(y)|} 
	\left[ \dfrac{v(x)^{q}}{u(x)^{q-1}} - \dfrac{v(y)^{q}}{u(y)^{q-1}} \right] 
	\\[4pt]
	&\leq 
	p^{+} G(1) |u(x) - u(y)|^{p^{-} - 2} (u(x) - u(y)) 
	\left[ \dfrac{v(x)^{q}}{u(x)^{q-1}} - \dfrac{v(y)^{q}}{u(y)^{q-1}} \right] 
	\\[4pt]
	&\leq 
	p^{+} G(1)  |v(x) - v(y)|^{q} \, |u(x) - u(y)|^{p^{-} - q}
	\\[4pt]
	&\leq 
	p^{+} G(1)^{\frac{q(p^{+} - p^{-}) + (p^{-})^{2}}{p^{+} p^{-}}} 
	\left( G(|v(x) - v(y)|) \right)^{\frac{q}{p^{-}}} 
	\left( G(|u(x) - u(y)|) \right)^{\frac{p^{-} - q}{p^{+}}}.
	\end{aligned}
	\end{equation*}
	
	\medskip
	\noindent
	Combining the estimates obtained in all four cases concludes the proof.
\end{proof}

\noindent
We recall the definition of strict ray-convexity.

\begin{definition} \label{defn1} \rm
	Let \( X \) be a real vector space, and let \( C \subset X \) be a nonempty convex cone.  
	A functional \( \mathcal{W} : C \to \mathbb{R} \) is said to be  ray-strictly convex (respectively, strictly convex) if
	\[
	\mathcal{W}((1 - t) v_1 + t v_2) \leq (1 - t)\, \mathcal{W}(v_1) + t\, \mathcal{W}(v_2) \quad \text{for all } v_1, v_2 \in C \text{ and all } t \in (0, 1),
	\]
	with strict inequality unless \( \frac{v_1}{v_2} \equiv c > 0 \) (unless \( v_1 \equiv v_2 \), respectively).
\end{definition}

\noindent
By Proposition~\ref{proposition1}, the set
\[
\dot{V}_{+}^{q} := \left\{ u : \Omega \to (0, \infty) \;\middle|\; u^{\frac{1}{q}} \in W^{s, G}_0(\Omega) \right\}, \quad \text{for } 1 < q \leq p^{-},
\]
is a convex cone. That is, for all \( \lambda \in (0, 1) \) and all \( f, g \in \dot{V}_{+}^{q} \), it holds that \( \lambda f + (1 - \lambda) g \in \dot{V}_{+}^{q} \).

\begin{proposition} \label{pro3}
	Let \( G \) be an \( N \)-function satisfying \textbf{(H4)}, and suppose that \( 1 < q \leq p^{-} \). Then the functional
	\[
	\mathcal{W} : \dot{V}_{+}^{q} \to \mathbb{R}_{+}, \quad 
	\mathcal{W}(w) := \int_{\mathbb{R}^N} \int_{\mathbb{R}^N} 
	G\left( \frac{ \left| w(x)^{1/q} - w(y)^{1/q} \right| }{ |x - y|^{s} } \right) 
	\frac{dx \, dy}{|x - y|^{N}},
	\]
	is ray-strictly convex on \( \dot{V}_{+}^{q} \). Moreover, if \( q \neq p^{-} \), then \( \mathcal{W} \) is strictly convex on \( \dot{V}_{+}^{q} \).
\end{proposition}

\begin{proof}
	According to Definition~\ref{defn1}, let \( w_1, w_2 \in \dot{V}_{+}^q \) and \( t \in [0, 1] \). Define \( w = (1 - t) w_1 + t w_2 \). Then, by Proposition~\ref{proposition1}, we have
	\begin{equation*}
	\mathcal{W}(w) \leq (1 - t)\, \mathcal{W}(w_1) + t\, \mathcal{W}(w_2).
	\end{equation*}
	Assume now that equality holds. Then, for almost every \( x, y \in \mathbb{R}^N \), we have
	\[
	G\left( \frac{ \left| w(x)^{1/q} - w(y)^{1/q} \right| }{ |x - y|^{s} } \right) 
	= (1 - t)\, G\left( \frac{ \left| w_{1}(x)^{1/q} - w_{1}(y)^{1/q} \right| }{ |x - y|^{s} } \right) 
	+ t\, G\left( \frac{ \left| w_{2}(x)^{1/q} - w_{2}(y)^{1/q} \right| }{ |x - y|^{s} } \right).
	\]
	Define an \( N \)-function \( H \) by \( H^{-1}(s) = G(s^{1/p^{-}}) \). Since the mapping \( s \mapsto G(s^{1/p^{-}}) \) is convex, it follows that \( H \) is concave. Then, by \cite[Proposition 4.1]{brasco2014convexity}, we obtain
	\[
	\left| w(x)^{1/q} - w(y)^{1/q} \right|^{p^{-}} 
	= (1 - t)\, \left| w_{1}(x)^{1/q} - w_{1}(y)^{1/q} \right|^{p^{-}} 
	+ t\, \left| w_{2}(x)^{1/q} - w_{2}(y)^{1/q} \right|^{p^{-}}.
	\]
	If \( p^{-} = q \), then we deduce
	\[
	\left| \| a \|_{\boldsymbol{\ell}^{q}} - \| b \|_{\boldsymbol{\ell}^{q}} \right|^{q}
	= \| a - b \|_{\boldsymbol{\ell}^{q}}^{q} \quad \text{for a.e. } x, y \in \mathbb{R}^N,
	\]
	where \( \| \cdot \|_{\boldsymbol{\ell}^{q}} \) denotes the \( \boldsymbol{\ell}^{q} \)-norm in \( \mathbb{R}^2 \), and
	\[
	a = \left( ((1 - t)w_1(x))^{1/q}, (t w_2(x))^{1/q} \right), \quad
	b = \left( ((1 - t)w_1(y))^{1/q}, (t w_2(y))^{1/q} \right).
	\]
	Since \( q > 1 \), there exists a constant \( c > 0 \) such that \( w_1 = c w_2 \) a.e. in \( \mathbb{R}^N \). Thus, \( \mathcal{W} \) is ray-strictly convex on \( \dot{V}_{+}^q \). On the other hand, if \( p^{-} \neq q \), then by the strict convexity of the mapping \( \tau \mapsto \tau^{\frac{p^{-}}{q}} \) on \( \mathbb{R}^{+} \), it follows that \( w_1 = w_2 \) a.e. in \( \mathbb{R}^N \). Therefore, \( \mathcal{W} \) is strictly convex on \( \dot{V}_{+}^q \).
\end{proof}

\noindent
We now establish the following extension of the D\'{\i}az and Saa inequality, involving the \( g \)-Laplacian operator, which plays a fundamental role in the analysis of the uniqueness of problem~\eqref{GP}.
\begin{lemma} \label{Lem2}
	Let \( G \) be an \( N \)-function satisfying \textbf{(H4)}, and suppose that \( 1 < q \leq p^{-} \). Then the following inequality holds in the sense of distributions:
	\begin{equation} \label{equ0}
	\begin{aligned}
	& \int_{\mathbb{R}^{N}} \int_{\mathbb{R}^{N}}  
	g\left( \frac{\left| u(x) - u(y)\right| }{|x - y|^{s}} \right)  \dfrac{ u(x) - u(y)}{\left|  u(x) - u(y)\right| }
	\left( \frac{u(x)^{q} - v(x)^{q}}{u(x)^{q - 1}} 
	- \frac{u(y)^{q} - v(y)^{q}}{u(y)^{q - 1}} \right) 
	\frac{dx \, dy}{|x - y|^{N + s}} \\[4pt]
	& \quad + \int_{\mathbb{R}^{N}} \int_{\mathbb{R}^{N}}  
	g\left( \frac{\left| v(x) - v(y)\right| }{|x - y|^{s}} \right)\dfrac{v(x) - v(y)}{\left|v(x) - v(y)\right| } 
	\left( \frac{v(x)^{q} - u(x)^{q}}{v(x)^{q - 1}} 
	- \frac{v(y)^{q} - u(y)^{q}}{v(y)^{q - 1}} \right) 
	\frac{dx \, dy}{|x - y|^{N + s}} \geq 0,
	\end{aligned}
	\end{equation}
	for all pairs \( u, v \in W^{s,G}_{0}(\Omega) \) such that \( u > 0 \) and \( v > 0 \) a.e. in \( \Omega \), and both \( u/v \) and \( v/u \) belong to \( L^{\infty}(\Omega) \).  Moreover, if equality holds in \eqref{equ0}, then the following assertions are satisfied:
	\begin{itemize}
		\item[(1)] \( u/v \equiv \text{const} > 0 \) a.e. in \( \Omega \).
		\item[(2)] If in addition \( q\neq p^{-}  \), then \( u \equiv v \) a.e. in \( \Omega \).
	\end{itemize}
\end{lemma}
\begin{proof}
	First, note that both integrals above are well-defined as Lebesgue integrals. Indeed, there exists a constant \( M > 0 \) such that
	\[
	0 \leq \frac{u^{q}}{v^{q-1}} = \left( \frac{u}{v} \right)^{q-1} u \leq M u, \quad \text{and} \quad 0 \leq \frac{v^{q}}{u^{q-1}} = \left( \frac{v}{u} \right)^{q-1} v \leq M v.
	\]
	By Lemmas~\ref{lemma1} and~\ref{lemma2}, it follows that \( \frac{u^{q}}{v^{q-1}} \in L^{G}(\mathbb{R}^{N}) \) and \( \frac{v^{q}}{u^{q-1}} \in L^{G}(\mathbb{R}^{N}) \), and both functions vanish outside \( \Omega \). Applying the Lagrange Mean Value Theorem yields:
	\begin{align*}
	&\left| \frac{u^q(x)}{v^{q-1}(x)} - \frac{u^q(y)}{v^{q-1}(y)} \right| 
	\leq \left| \frac{u^q(x)}{v^{q-1}(x)} - \frac{u^q(y)}{v^{q-1}(x)} \right| + \left| \frac{u^q(y)}{v^{q-1}(x)} - \frac{u^q(y)}{v^{q-1}(y)} \right|\\
	&\quad \leq \frac{q\, \max\left\lbrace u^{q-1}(x), u^{q-1}(y) \right\rbrace}{v^{q-1}(x)} \left| u(x) - u(y) \right|   + \frac{u^{q}(y)\, (q-1)\, \max\left\lbrace v^{q-2}(x), v^{q-2}(y) \right\rbrace}{v^{q-1}(x)\, v^{q-1}(y)} \left| v(x) - v(y) \right| \\
	&\quad \leq C\left( q, \left\| u/v\right\|_{L^{\infty}(\Omega)} \right) \left(1/2\left| u(x) - u(y) \right| + 1/2\left| v(x) - v(y) \right| \right),
	\end{align*}
	for some constant \( C > 0 \). Therefore, again by Lemmas~\ref{lemma1} and~\ref{lemma2}, we conclude that \( \frac{u^{q}}{v^{q-1}}, \frac{v^{q}}{u^{q-1}} \in W^{s,G}_{0}(\Omega) \).
	
	\noindent
	Now, consider \( u, v \in W^{s, G}_0(\Omega) \) such that \( u > 0 \) and \( v > 0 \) a.e.\ in \( \Omega \), and let \( \theta \in (0,1) \). Define
	\[
	w := (1-\theta) u^q + \theta v^q.
	\]
	By Proposition~\ref{proposition1}, it is easily seen that \( w \in \dot{V}_{+}^q \). Hence, by Proposition~\ref{pro3}, the function
	\[
	\theta \mapsto \Phi(\theta) := \mathcal{W}(w) = \mathcal{W}\left((1 - \theta) u^q + \theta v^q \right)
	\]
	is convex and differentiable on \( [0, 1] \). For \( \theta \in (0,1) \), we compute:
{\small 	\begin{align*}
	\Phi'(\theta)
	&= \int_{\mathbb{R}^{2N} \setminus (\Omega^c \times \Omega^c)}
	g\left( \frac{\left| w(x)^{1/q} - w(y)^{1/q}\right| }{|x - y|^{s}} \right)\dfrac{w(x)^{1/q} - w(y)^{1/q}}{\left| w(x)^{1/q} - w(y)^{1/q}\right| }
	\left( \frac{v(x)^q - u(x)^q}{w(x)^{1 - 1/q}} - \frac{v(y)^q - u(y)^q}{w(y)^{1 - 1/q}} \right)
	\frac{dx \, dy}{|x - y|^{N + s}}.
	\end{align*}}
	Using the convexity of \( \Phi \), and the identities \( w = u^q \) when \( \theta = 0 \), and \( w = v^q \) when \( \theta = 1 \), we obtain:
	\[
	\Phi'(0) = \lim_{\theta \to 0^+} \Phi'(\theta) \leq \lim_{\theta \to 1^-} \Phi'(\theta) = \Phi'(1),
	\]
	which is equivalent to inequality~\eqref{equ0}. Finally, assume that equality holds in~\eqref{equ0}. Since \( \Phi' : (0,1) \to \mathbb{R} \) is monotone, it follows that \( \Phi'(\theta) = \text{const} \) on \( (0,1) \), and hence \( \Phi : [0,1] \to \mathbb{R} \) is linear:
	\[
	\Phi(\theta) = \mathcal{W}(w) = (1-\theta) \Phi(0) + \theta \Phi(1) = (1-\theta) \mathcal{W}(u^q) + \theta \mathcal{W}(v^q),
	\]
	for all \( \theta \in [0,1] \). We deduce that \( u \equiv c v \) for some \( c > 0 \). Moreover, if \( q \neq p^{-} \), then \( u \equiv v \) by Proposition~\ref{pro3}.
\end{proof}
\begin{remark}
	An alternative proof of Theorem~4.2 can be obtained by employing the \( G \)-fractional Picone inequality stated in Proposition~\ref{proposition2}, together with Young's inequality. This approach leads directly to the derivation of \eqref{equ0}. We emphasize that the \( G \)-fractional Picone inequality, being a pointwise estimate, is stronger than the D\'{\i}az and Saa inequality \eqref{equ0}, and is of independent interest.
\end{remark}

\noindent
Before presenting our main results, we define the set \( \textbf{FC} \), which consists of the following class of functions:\\[4pt]
\[
\textbf{FC} = \left\lbrace 
\Psi : (0, \infty) \to (0, \infty) \;\middle|\;
\begin{array}{l}
\text{(i) } \Psi \text{ is of class } C^{1} \text{ and strictly convex}, \\[4pt]
\text{(ii) } \Psi' \text{ is increasing on } (0, \infty), \\[4pt]
\text{(iii) there exist constants } \theta_1, \theta_2 \geq 0 \text{ such that} \\[2pt]
\qquad \theta_1 \dfrac{\Psi(x)}{x} \leq \Psi'(x) \leq \theta_2 \dfrac{\Psi(x)}{x} \quad \forall x > 0.
\end{array}
\right\rbrace.
\]

\subsection{Statements of main results} 
We first observe that, under the assumptions imposed on the function \( g \), the mapping \( g : \mathbb{R} \to \mathbb{R} \) is injective. Therefore, by the definition of the complementary function \( \overline{G} \) associated with \( G \), we obtain the following estimate:
\begin{equation}\label{3equ03}
(p^{-} - 1)\, G(t) \leq \overline{G}(g(t)) \leq (p^{+} - 1)\, G(t) , \quad\text{ for all }  t \geq 0.
\end{equation}
Accordingly, we introduce the notions of weak sub-solutions, super-solutions, and solutions to problem~\eqref{GP}:

\begin{definition}\label{definition1}
We say that \( u \in W^{s, G}_{\mathrm{loc}}(\Omega) \) is a \textbf{weak super-solution} to \eqref{GP} if the following conditions hold:
\begin{enumerate}
	\item[(i)] There exists a function \( \Psi \in \textbf{FC} \) such that \( \Psi(u) \in W^{s, G}_{0}(\Omega) \).
	
	\item[(ii)] For every compact set \( K \Subset \Omega \), there exists a constant \( C(K) > 0 \) such that
	\[
	u \geq C(K) \quad \text{a.e. in } K.
	\]
	
	\item[(iii)] For all test functions \( \varphi \in W^{s, G}_{0}(\Omega) \cap L^{\infty}_{c}(\Omega) \) with \( \varphi \geq 0 \), the following inequality holds:
\[
\int_{\mathbb{R}^{N}} \int_{\mathbb{R}^{N}} 
g\left( \frac{\left| u(x) - u(y) \right|}{\left| x - y \right|^{s}} \right)
\dfrac{u(x) - u(y)}{\left| u(x) - u(y) \right|} 
\dfrac{\varphi(x) - \varphi(y)}{\left| x - y \right|^{N + s}} \, dx \, dy 
\geq \int_{\Omega} \left( f(x)\, u^{-\alpha} + k(x)\, u^{\beta} \right) \varphi(x) \, dx.
\]
\end{enumerate}
	If \( u \) satisfies the reverse inequality in (iii), then it is called a \textbf{weak sub-solution} to problem~\eqref{GP}. A function \( u \) that satisfies the requirements of both a weak sub and super-solution is referred to as a \textbf{weak solution} to~\eqref{GP}.
\end{definition}

\begin{remark}\label{remark2}
An important observation regarding Definition~\ref{definition1} is that the solution \( u \) generally does not belong to the space \( W^{s, G}_{0}(\Omega) \). Moreover, it is important to emphasize that the space \( W^{s, G}_{\mathrm{loc}}(\Omega) \) does not admit a trace operator. For this reason, we adopt the following interpretation of the Dirichlet boundary condition in a generalized sense (see \cite[Definition 1.3]{canino2017nonlocal}):

\medskip
\noindent
\textbf{``}We say that \( u \leq 0 \) on \( \partial \Omega \) if \( u = 0 \) in \( \mathbb{R}^{N} \setminus \Omega \) and \( (u - \epsilon)^{+} \in W^{s, G}_{0}(\Omega) \) for every \( \epsilon > 0 \). Moreover, we say that \( u = 0 \) on \( \partial \Omega \) if \( u \geq 0 \) and \( u \leq 0 \) on \( \partial \Omega \).\textbf{''}

\medskip	
\noindent It is crucial to note that Definition \ref{definition1} - \textbf{(i)} ensures that the solution satisfies the conditions of this definition. In particular, since \( \Psi(u) \in W^{s,G}_{0}(\Omega) \) for a function \( \Psi \in \textbf{FC} \), and given any \( \epsilon > 0 \), we define
\[
\mathcal{S}_{\epsilon} := \operatorname{supp} (u - \epsilon)^{+}, \quad \text{and} \quad \mathcal{Q}_{\epsilon} := (\mathbb{R}^{N} \times \mathbb{R}^{N}) \setminus (\mathcal{S}_{\epsilon}^{c} \times \mathcal{S}_{\epsilon}^{c}).
\]
Then, invoking \cite[Lemma 3.11]{bal2024singular} and applying Lemma~\ref{lemma2}, we obtain
\begin{align*}
\min\left\{ \left\| (u - \epsilon)^{+} \right\|_{W^{s,G}_{0}(\Omega)}^{p^{-}},\, \left\| (u - \epsilon)^{+} \right\|_{W^{s,G}_{0}(\Omega)}^{p^{+}} \right\}
& \leq \int_{\mathbb{R}^{N}}  \int_{\mathbb{R}^{N}}  G\left( \frac{\left| (u - \epsilon)^{+}(x) - (u - \epsilon)^{+}(y) \right|}{\left| x - y \right|^{s}} \right) \frac{dx\,dy}{\left| x - y \right|^{N}} \\[4pt]
& \leq \iint_{\mathcal{Q}_{\epsilon}} G\left( \frac{\left| u(x) - u(y) \right|}{\left| x - y \right|^{s}} \right) \frac{dx\,dy}{\left| x - y \right|^{N}} \\[4pt]
& \leq C(\epsilon) \iint_{\mathcal{Q}_{\epsilon}} G\left( \frac{\left| \Psi(u(x)) - \Psi(u(y)) \right|}{\left| x - y \right|^{s}} \right) \frac{dx\,dy}{\left| x - y \right|^{N}} \\[4pt]
& \leq \max\left\{ \left\| \Psi(u) \right\|_{W^{s,G}_{0}(\Omega)}^{p^{-}},\, \left\| \Psi(u) \right\|_{W^{s,G}_{0}(\Omega)}^{p^{+}} \right\} < \infty.
\end{align*}
As a consequence, we conclude that \( (u - \epsilon)^{+} \in W^{s,G}_{0}(\Omega) \) for every \( \epsilon > 0 \).
\end{remark}

\noindent
Our main result concerning the weak comparison principle for problem~\eqref{GP} in $W^{s, G}_{\text{loc}}(\Omega)$ is formulated below.
\begin{theorem}\label{Theorem1}
\noindent
Let \( G \) be an \( N \)-function satisfying conditions  \textbf{(H1)}–\textbf{(H4)}. Suppose that \( k \in L^{\overline{K}}(\Omega) \), where \( \overline{K} \) denotes the complementary function of the \( N \)-function \( K \), defined by
\[
K(t) := G_{*}\left(t^{\frac{1}{\beta + 1}}\right).
\]
In addition, assume that one of the following  assumptions holds:
\begin{itemize}
	\item[\textbf{(A1)}] \( \alpha > 1 \) and \( f \in L^{1}(\Omega) \).
	\item[\textbf{(A2)}] \( \alpha = 1 \) and \( f \in L^{\overline{G_{*}}}(\Omega) \).
	\item[\textbf{(A3)}] \( \alpha < 1 \) and \( f \in L^{\overline{\mathfrak{F}}}(\Omega) \), where \( \overline{\mathfrak{F}} \) denotes the complementary function of the \( N \)-function \( \mathfrak{F} \), defined by
\begin{equation} \label{3equ0}
\mathfrak{F}(t) = G_{*} \left( t^{\frac{1}{1 - \alpha}} \right).
\end{equation}
\end{itemize}
Let \( \underline{u}, \overline{u} \in W^{s, G}_{\mathrm{loc}}(\Omega) \) be a weak sub-solution and super-solution, respectively, of problem~\eqref{GP} in the sense of Definition~\ref{definition1}. Then it follows that \( \underline{u} \leq \overline{u} \) almost everywhere in \( \Omega \).
\end{theorem}

\noindent
We point out that assuming \( \underline{u}, \overline{u} \in W^{s, G}_{0}(\Omega) \) allows for a substantial simplification of the proof of Theorem~\ref{Theorem1}. In what follows, we establish a result tailored to problem~\eqref{GGP}, after first introducing the definitions of sub-solutions, super-solutions, and solutions in our framework.
\begin{definition}\label{definition2}
	We define a nonnegative function \( u \in W^{s, G}_{0}(\Omega) \) as a weak super-solution to \eqref{GGP} if it satisfies the following conditions:
	\begin{enumerate}
		\item[(i)] \( F(\cdot, u(\cdot)) \varphi(\cdot) \in L^{1}(\Omega) \) for all \( \varphi \in W^{s, G}_{0}(\Omega) \).
		\item[(ii)] For all \( \varphi \in W^{s, G}_{0}(\Omega) \) with \( \varphi \geq 0 \), we have
\begin{equation*}
\int_{\mathbb{R}^{N}} \int_{\mathbb{R}^{N}} g\left( \frac{\left| u(x) - u(y) \right|}{\left| x - y \right|^{s}} \right) 
\dfrac{u(x) - u(y)}{\left| u(x) - u(y) \right|} 
\dfrac{\varphi(x) - \varphi(y)}{\left| x - y \right|^{N + s}} \, dx \, dy 
\geq \int_{\Omega} F(x, u) \varphi(x) \, dx.
\end{equation*}
	\end{enumerate}
	If \( u \) satisfies the reverse inequality, it is called a weak sub-solution of \eqref{GGP}. A function \( u \) that satisfies both the weak sub- and super-solution conditions is termed a weak solution of \eqref{GGP}.
\end{definition}

\begin{theorem}\label{Theorem2}
	Assume \( F \) satisfies \textbf{(F1)} and \textbf{(F2)}. Let \( \underline{u}, \overline{u} \in W^{s, G}_{0}(\Omega) \) be, respectively, a weak sub-solution and super-solution to the problem \eqref{GGP} in the sense of Definition \ref{definition2}. Then, \( \underline{u} \leq \overline{u} \)  almost everywhere in \( \Omega \).
\end{theorem}

\noindent
As a direct consequence of the weak comparison principle established in Theorem~\ref{Theorem1}, we obtain the following uniqueness result:

\begin{corollary}\label{corollary}
	Let \( G \) be an \( N \)-function satisfying hypotheses \textbf{(H1)}–\textbf{(H4)}. Suppose that the assumptions \textbf{(A1)}–\textbf{(A3)} stated in Theorem~\ref{Theorem1} hold. Then, any weak solution to problem~\eqref{GP} in the sense of Definition~\ref{definition1} is unique.
\end{corollary}

\begin{remark}
	This result complements the analysis developed in paper~\cite{bal2024singular}, where the existence and regularity of solutions were investigated.
\end{remark}

\noindent
We now highlight a direct consequence of the uniqueness result:

\begin{corollary}\label{corollary2}
	Let \( G \) be an \( N \)-function satisfying hypotheses \textbf{(H1)}–\textbf{(H4)}. Suppose that the conditions \textbf{(A1)}–\textbf{(A3)} of Theorem~\ref{Theorem1} are fulfilled, and let \( u \) denote the unique solution, if it exists, to problem~\eqref{GP} in the sense of Definition~\ref{definition1}. Assume further that the domain \( \Omega \) is symmetric with respect to the hyperplane
	\[
	\mathcal{H}^{\upsilon}_{\lambda} := \left\{ x \in \mathbb{R}^{N} : x \cdot \upsilon = \lambda \right\}, \quad \lambda \in \mathbb{R}, \quad \upsilon \in \mathbf{S}^{N-1}.
	\]
	If, in addition, both \( f \) and \( g \) are symmetric with respect to \( \mathcal{H}^{\upsilon}_{\lambda} \), then the solution \( u \) inherits this symmetry. In particular, if \( \Omega \) is a ball or an annulus centered at the origin, and both \( f \) and \( g \) are radially symmetric, then the solution \( u \) is radially symmetric as well.
\end{corollary}

\section{Proof of the Main Results}\label{sc3}

\noindent
In this section, we begin with the proof of our first main result.
\begin{proof}[\textbf{Proof of Theorem~\ref{Theorem1}}]
First, we define the following convex and closed subset of $W^{s, G}_{0}(\Omega)$:
\[
\mathcal{K} := \left\lbrace \phi \in W^{s, G}_{0}(\Omega)\, :\, 0 \leq \phi \leq \overline{u} \text{ a.e. in } \Omega \right\rbrace.
\]
Then, for each fixed $\epsilon \in (0,1)$, we introduce the functional $\mathcal{J}_{\epsilon} : \mathcal{K} \to \mathbb{R}$ defined by
\begin{align*}
\mathcal{J}_{\epsilon}(w) &:= \int_{\mathbb{R}^{N}} \int_{\mathbb{R}^{N}} G\left( \frac{\left| w(x) - w(y) \right|}{\left| x - y \right|^{s}} \right) \frac{dx \, dy}{\left| x - y \right|^{N}}
- \int_{\Omega} \mathbb{G}_{\epsilon}(x, w) \, dx,
\end{align*}
where
\begin{equation} \label{3comequ13}
\mathbb{G}_{\epsilon}(x, w) := \int_{0}^{w} \left[ f(x)(s + \epsilon)^{-\alpha} + k(x)(s + \epsilon)^{\beta} \right] \, ds.
\end{equation}
At this stage, we aim to establish three fundamental claims. We begin with the following:

\begin{claim}
	For any \( \epsilon > 0 \), there exists a minimizer \( w_{0} \in \mathcal{K} \) of the functional \( \mathcal{J}_{\epsilon} \) such that, for every \( \psi \in w_{0} + (W^{s, G}_{0}(\Omega) \cap L^{\infty}_{c}(\Omega)) \) with \( \psi \in \mathcal{K} \), the following variational inequality holds:
\begin{equation} \label{3comequ8}
\begin{aligned}
\int_{\mathbb{R}^{N}} \int_{\mathbb{R}^{N}}& g\left( \frac{\left| w_{0}(x) - w_{0}(y)\right| }{\vert x - y \vert^{s}} \right) \dfrac{w_{0}(x) - w_{0}(y)}{\left| w_{0}(x) - w_{0}(y) \right|} \cdot \frac{[(\psi - w_{0})(x) - (\psi - w_{0})(y)]}{\vert x - y \vert^{N + s}} \, dx \, dy \\[4pt]
&\quad \geq \int_{\Omega} \left[ f(x) (w_{0} + \epsilon)^{-\alpha} + k(x) (w_{0} + \epsilon)^{\beta} \right] (\psi - w_{0}) \, dx.
\end{aligned}
\end{equation}
\end{claim}

\noindent 
First, by combining Lemma~\ref{lemma2} with Lemma~\ref{lemma4} (see also Remark~\ref{remark1}), we deduce that the functional \( \mathcal{J}_{\epsilon} \) is well defined on \( \mathcal{K} \). Next, we establish that \( \mathcal{J}_{\epsilon} \) is coercive on \( \mathcal{K} \). Indeed, let \( w \in \mathcal{K} \). By applying H\"{o}lder's inequality (see Lemma~\ref{lemma3}) together with Remark~\ref{remark1}, and assuming that \( k \in L^{\overline{K}}(\Omega) \), where \( \overline{K} \) denotes the complementary function of the \( N \)-function \( K \), defined by
\[
K(t) := G_{*}\left(t^{\frac{1}{\beta + 1}}\right),
\]
we obtain the following conclusion depending on the value of \( \alpha \), provided that \( \left\| w \right\|_{W^{s,G}_0(\Omega)} \) is sufficiently large.\\[4pt]
\textbf{Case 1.} Assume that \( 0 < \alpha < 1 \). Then, we observe that
\begin{equation*} 
\int_{\Omega} f u^{1 - \alpha} \, dx \leq 2 \left\| f \right\|_{L^{\overline{\mathfrak{F}}}(\Omega)} \left\| u^{1 - \alpha} \right\|_{L^{\mathfrak{F}}(\Omega)},
\end{equation*}
where the modular norm is given by
\begin{equation*} 
\begin{aligned}
\left\| u^{1 - \alpha} \right\|_{L^{\mathfrak{F}}(\Omega)} 
&= \inf \left\lbrace \lambda > 0 : \int_{\Omega} \mathfrak{F}\left( \frac{u^{1 - \alpha}}{\lambda} \right) \, dx \leq 1 \right\rbrace \\[4pt]
&= \left[ \inf \left\lbrace \lambda > 0 : \int_{\Omega} G_{*}\left( \frac{u}{\lambda} \right) \, dx \leq 1 \right\rbrace \right]^{1 - \alpha} \\[4pt]
&= \left\| u \right\|_{L^{G_{*}}(\Omega)}^{1 - \alpha},
\end{aligned}
\end{equation*}
with \( \mathfrak{F} \) defined in \eqref{3equ0}. In addition, it holds that
\[
\left\| u^{1 + \beta} \right\|_{L^{K}(\Omega)} = \left\| u \right\|_{L^{G_{*}}(\Omega)}^{1 + \beta}.
\]
Hence, we obtain
\begin{equation*} 
\begin{aligned}
\mathcal{J}_{\epsilon}(w)  
&\geq \left\| w \right\|_{W^{s, G}_{0}(\Omega)}^{p^{-}} 
- C \left[  
\left\| f \right\|_{L^{\overline{\mathfrak{F}}}(\Omega)} 
\left\| w \right\|_{L^{G_{*}}(\Omega)}^{1 - \alpha}
+ \left\| k \right\|_{L^{\overline{K}}(\Omega)} 
\left\| w \right\|_{L^{G_{*}}(\Omega)}^{1 + \beta} 
+ 1 
\right] \\[4pt]
&\geq \left\| w \right\|_{W^{s, G}_{0}(\Omega)}^{p^{-}} 
\left[  1
- C\left( 
\left\| f \right\|_{L^{\overline{\mathfrak{F}}}(\Omega)} 
\left\| w \right\|_{W^{s, G}_{0}(\Omega)}^{1 - \alpha - p^{-}} 
+ \left\| k \right\|_{L^{\overline{K}}(\Omega)} 
\left\| w \right\|_{W^{s, G}_{0}(\Omega)}^{1 + \beta - p^{-}} 
+ \left\| w \right\|_{W^{s, G}_{0}(\Omega)}^{-p^{-}} 
\right) 
\right] .
\end{aligned}
\end{equation*}
\textbf{Case 2:} When \( \alpha > 1 \), we obtain
\begin{equation*} 
\begin{aligned}
\mathcal{J}_{\epsilon}(w) 
&\geq \left\| w \right\|_{W^{s, G}_0(\Omega)}^{p^{-}} 
- C \left( \left\| k \right\|_{L^{\overline{K}}(\Omega)} \left\| w \right\|_{L^{G_{*}}(\Omega)}^{1 + \beta} + 1 \right) \\[4pt]
&\geq \left\| w \right\|_{W^{s, G}_0(\Omega)}^{p^{-}} 
\left[ 1 - C \left( \left\| k \right\|_{L^{\overline{K}}(\Omega)} \left\| w \right\|_{W^{s, G}_0(\Omega)}^{1 + \beta - p^{-}} 
+ \left\| w \right\|_{W^{s, G}_0(\Omega)}^{-p^{-}} \right) \right].
\end{aligned}
\end{equation*}
\textbf{Case 3:} For \( \alpha = 1 \), we observe that there exists a constant \( C(\epsilon) > 0 \) such that
\begin{equation*}
\ln(w + \epsilon) \leq C(\epsilon)(w + \epsilon).
\end{equation*}
As a consequence, we obtain
\begin{equation*}
\begin{aligned}
\mathcal{J}_{\epsilon}(w) &\geq \left\| w \right\|^{p^{-}}_{W^{s, G}_{0}(\Omega)} 
- C\left( \left\| f \right\|_{L^{\overline{G_{*}}}(\Omega)} \left\| w \right\|_{L^{G_{*}}(\Omega)}
+ \left\| k \right\|_{L^{\overline{K}}(\Omega)} \left\| w \right\|^{1 + \beta}_{L^{G_{*}}(\Omega)} + 1 \right) \\[4pt]
&\geq \left\| w \right\|^{p^{-}}_{W^{s, G}_{0}(\Omega)} 
\left[  1 - C\left( \left\| f \right\|_{L^{\overline{G_{*}}}(\Omega)} \left\| w \right\|^{1 - p^{-}}_{W^{s, G}_{0}(\Omega)}
+ \left\| k \right\|_{L^{\overline{K}}(\Omega)} \left\| w \right\|^{1 + \beta - p^{-}}_{W^{s, G}_{0}(\Omega)} 
+ \left\| w \right\|^{-p^{-}}_{W^{s, G}_{0}(\Omega)} \right)\right] .
\end{aligned}
\end{equation*}
Therefore, we conclude that \( \mathcal{J}_{\epsilon}(w) \to \infty \) as \( \left\| w \right\|_{W^{s, G}_{0}(\Omega)} \to \infty \), in all the cases considered above. Moreover, the energy functional \( \mathcal{J}_{\epsilon} \) is weakly lower semi-continuous on \( \mathcal{K} \). In particular, let \( (w_n) \subset \mathcal{K} \) be a sequence that converges weakly to some \( w \in \mathcal{K} \) as \( n \to \infty \). Hence, we have
	\begin{equation} \label{3equ1}
	\begin{gathered}
	\begin{aligned}
	\liminf_{n\to \infty} \mathcal{J}_{\epsilon}(u_{n}) \geq 	\liminf_{n\to \infty} \int_{\mathbb{R}^{N}} \int_{\mathbb{R}^{N}} G\left( \frac{\left| u_{n}(x) - u_{n}(y) \right|}{\left| x - y \right|^{s}} \right) \frac{dx \, dy}{\left| x - y \right|^{N}} -\limsup_{n \to \infty} \int_{\Omega} \mathbb{G}_{\epsilon}(x, u_{n}) \, dx.
	\end{aligned}
	\end{gathered}
	\end{equation}
On the other hand, based on the different cases of \( \alpha \), we have
\begin{equation*}
\begin{aligned}
&\int_{\Omega} f(x) (w_n + \epsilon)^{1 - \alpha} \, dx 
\leq C(\epsilon) \left\| f \right\|_{L^{1}(\Omega)} 
\leq C(\epsilon), \quad &&\text{if } \alpha > 1, \\[4pt]
&\int_{\Omega} f(x) \ln(w_n + \epsilon) \, dx 
\leq C(\epsilon) \left\| f \right\|_{L^{\overline{G_{*}}}(\Omega)} \left( 
\left\| w_n \right\|_{W^{s, G}_{0}(\Omega)} + 1\right) 
\leq C(\epsilon), \quad &&\text{if } \alpha = 1, \\[4pt]
&\int_{\Omega} f(x) (w_n + \epsilon)^{1 - \alpha} \, dx 
\leq C(\epsilon) \left\| f \right\|_{L^{\overline{\mathfrak{F}}}(\Omega)} 
\left( \left\| w_n \right\|_{W^{s, G}_{0}(\Omega)}^{1 - \alpha} + 1 \right) 
\leq C(\epsilon), \quad &&\text{if } \alpha < 1.
\end{aligned}
\end{equation*}
Additionally, we find that
\begin{equation} \label{3equ2}
\begin{aligned}
\int_{\Omega} k(x) (w_n + \epsilon)^{\beta + 1} \, dx 
&\leq C(\epsilon) \left\| k \right\|_{L^{\overline{K}}(\Omega)} 
\left( \left\| w_n \right\|_{W^{s, G}_{0}(\Omega)}^{\beta + 1} + 1 \right)  \leq C(\epsilon),
\end{aligned}
\end{equation}
where \( C(\epsilon) > 0 \) is a constant independent of \( n \).  Then, by applying Vitali's Convergence Theorem, we get that
\[
\int_{\Omega} \int_{0}^{w_n} f(x)(s + \epsilon)^{-\alpha} \, ds \, dx 
\longrightarrow \int_{\Omega} \int_{0}^{w} f(x)(s + \epsilon)^{-\alpha} \, ds \, dx 
\quad \text{as } n \to \infty,
\]
and
\[
\int_{\Omega} \int_{0}^{w_n} k(x)(s + \epsilon)^{\beta} \, ds \, dx 
\longrightarrow \int_{\Omega} \int_{0}^{w} k(x)(s + \epsilon)^{\beta} \, ds \, dx 
\quad \text{as } n \to \infty.
\]
By combining all the aforementioned results with Fatou's Lemma and relation \eqref{3equ1}, we deduce that
\begin{equation*}
\liminf_{n \to \infty} \mathcal{J}_{\epsilon}(w_n) \geq \mathcal{J}_{\epsilon}(w).
\end{equation*}
Therefore, in light of the above properties,  \( \mathcal{J}_{\epsilon} \) admits a minimizer \( w_{0} \) in the convex and closed set \( \mathcal{K} \).\\[4pt]
Now, let us consider a function \( \psi \in w_{0} + (W^{s, G}_{0}(\Omega) \cap L_{c}^{\infty}(\Omega)) \) such that \( \psi \in \mathcal{K} \). Define the function:
\[
\xi : [0, 1] \to \mathbb{R}, \quad \xi(\tau) = \mathcal{J}_{\epsilon}(\tau \psi + (1 - \tau) w_{0}).
\]
Then, we have
{\small \begin{equation}\label{3equ5}
\begin{gathered}
\begin{aligned}
& 0 \leq \dfrac{\xi(\tau) - \xi(0)}{\tau} = \dfrac{\mathcal{J}_{\epsilon}(w_{0} + \tau(\psi - w_{0})) - \mathcal{J}_{\epsilon}(w_{0}) }{s}\\[4pt]
&= \underbrace{ \dfrac{1}{\tau} \int_{\mathbb{R}^{N}} \int_{\mathbb{R}^{N}} \left[ 
G\left( \dfrac{\left| (w_{0} + \tau(\psi - w_{0}))(x) - (w_{0} + \tau(\psi - w_{0}))(y) \right|}{\left| x - y \right|^{s}} \right) 
- G\left( \dfrac{\left| w_{0}(x) - w_{0}(y) \right|}{\left| x - y \right|^{s}} \right) 
\right] \dfrac{dx\, dy}{\left| x - y \right|^{N}}  }_{\textbf{I}} \\[4pt]
&\quad -\dfrac{1}{\tau} \int_{\Omega} \left[ G_{\epsilon}(x, w_{0} + \tau(\psi - w_{0})) - G_{\epsilon}(x, w_{0}) \right] dx.
\end{aligned}
\end{gathered}
\end{equation}}

\noindent 
To pass to the limit in term \textbf{I}, we first recall the Lagrange Mean Value Theorem, which gives:
{\small \begin{align*}
& G\left( \frac{\left| (w_{0} + \tau(\psi - w_{0}))(x) - (w_{0} + \tau(\psi - w_{0}))(y) \right|}{\left| x - y \right|^{s}} \right) 
- G\left( \frac{\left| w_{0}(x) - w_{0}(y) \right|}{\left| x - y \right|^{s}} \right) \\[4pt]
&= g\left( \frac{\left| (w_{0} + C(\tau)(\psi - w_{0}))(x) - (w_{0} + C(\tau)(\psi - w_{0}))(y) \right|}{\left| x - y \right|^{s}} \right) \frac{ \left( w_{0} + C(\tau)(\psi - w_{0}) \right)(x) - \left( w_{0} + C(\tau)(\psi - w_{0}) \right)(y)}{\left| \left( w_{0} + C(\tau)(\psi - w_{0}) \right)(x) - \left( w_{0} + C(\tau)(\psi - w_{0}) \right)(y) \right|} \\[4pt]
&\quad \cdot \frac{ \tau \left( (\psi - w_{0})(x) - (\psi - w_{0})(y) \right)}{ \left| x - y \right|^{s} },
\end{align*}}

\noindent 
where \( C(\tau) \in (0, \tau) \). Then we have
\begin{equation}\label{3equ4}
\begin{gathered}
\begin{aligned}
&\lim_{\tau \to 0} \frac{1}{s} \left[ 
G\left( \frac{\left| (w_{0} + \tau(\psi - w_{0}))(x) - (w_{0} + \tau(\psi - w_{0}))(y) \right|}{\left| x - y \right|^{s}} \right) 
- G\left( \frac{\left| w_{0}(x) - w_{0}(y) \right|}{\left| x - y \right|^{s}} \right) 
\right]  \\[4pt] 
& = g\left( \frac{\left| w_{0}(x) - w_{0}(y)\right| }{\vert x - y \vert^{s}} \right) \frac{w_{0}(x) - w_{0}(y)}{\left| w_{0}(x) - w_{0}(y) \right|} \cdot \frac{[(\psi - w_{0})(x) - (\psi - w_{0})(y)]}{\vert x - y \vert^{s}}.
\end{aligned}
\end{gathered}
\end{equation}
On the other hand, for \( \tau \leq \tau_{0} \), by Young's inequality (see \eqref{equ3}), \eqref{3equ03}, and Lemma \ref{lemma1}, we have:
{\small \begin{align*}
&\frac{1}{\tau} \left[ G\left( \frac{\left| (w_{0} + \tau(\psi - w_{0}))(x) - (w_{0} + \tau(\psi - w_{0}))(y) \right|}{\left| x - y \right|^{s}} \right) 
- G\left( \frac{\left| w_{0}(x) - w_{0}(y) \right|}{\left| x - y \right|^{s}} \right)\right] \\[4pt] 
& \leq g\left( \frac{\left| (w_{0} + C(\tau)(\psi - w_{0}))(x) - (w_{0} + C(\tau)(\psi - w_{0}))(y) \right|}{\left| x - y \right|^{s}} \right) 
\cdot \frac{ (\psi - w_{0})(x) - (\psi - w_{0})(y) }{ \left| x - y \right|^{s} } \\[4pt] 
& \leq \overline{G}\left(g\left( \frac{\left| (w_{0} + C(\tau)(\psi - w_{0}))(x) - (w_{0} + C(\tau)(\psi - w_{0}))(y) \right|}{\left| x - y \right|^{s}} \right) \right) +  G\left(\frac{ \left| (\psi - w_{0})(x) - (\psi - w_{0})(y)\right| }{ \left| x - y \right|^{s} } \right) \\[4pt] 
& \leq (p^{+} - 1) G\left( \frac{\left| (w_{0} + C(\tau)(\psi - w_{0}))(x) - (w_{0} + C(\tau)(\psi - w_{0}))(y) \right|}{\left| x - y \right|^{s}} \right)+  G\left(\frac{ \left| (\psi - w_{0})(x) - (\psi - w_{0})(y)\right| }{ \left| x - y \right|^{s} } \right) \\[4pt] 
& \leq  C(\tau_{0}) \left[ G\left(\frac{\left|  w_{0}(x) - w_{0}(y)\right| }{ \left| x - y \right|^{s} } \right)  + G\left(\frac{\left|  (\psi - w_{0})(x) - (\psi - w_{0})(y)\right| }{ \left| x - y \right|^{s} } \right) \right] \in L^{1}(\Omega).
\end{align*}}

\noindent 
This fact, together with \eqref{3equ4} and the dominated convergence theorem, implies that
\begin{align*}
\lim_{\tau \to 0^{+}} \textbf{I} = \int_{\mathbb{R}^{N}} \int_{\mathbb{R}^{N}} g\left( \frac{\left| w_{0}(x) - w_{0}(y)\right| }{\left| x - y \right|^{s}} \right)
\frac{w_{0}(x) - w_{0}(y)}{\left| w_{0}(x) - w_{0}(y) \right|}
\cdot \frac{(\psi - w_{0})(x) - (\psi - w_{0})(y)}{\left| x - y \right|^{N+s}} \, dx \, dy.
\end{align*}
Finally, passing to the limit as \( \tau \to 0^{+} \) in \eqref{3equ5}, we obtain that \eqref{3comequ8} is true.

	\begin{claim}
		For every $ \psi \in W^{s, G}_{0}(\Omega)\cap L^{\infty}(\Omega),$ with $  \psi \geq 0 ,$ a.e. in \( \Omega \), it holds that:
		\begin{equation}\label{3comequ10}
		\begin{gathered}
		\begin{aligned}
		\int_{\mathbb{R}^{N}} \int_{\mathbb{R}^{N}}& g\left( \frac{\left| w_{0}(x) - w_{0}(y)\right| }{\vert x - y \vert^{s}} \right) \dfrac{w_{0}(x) - w_{0}(y)}{\left| w_{0}(x) - w_{0}(y) \right|} \cdot \frac{\psi(x) - \psi(y)}{\vert x - y \vert^{N + s}} \, dx \, dy \\[4pt]
		&\quad \geq \int_{\Omega} \left[ f(x) (w_{0} + \epsilon)^{-\alpha} + k(x) (w_{0} + \epsilon)^{\beta} \right] \psi\, dx.
		\end{aligned}
		\end{gathered}
		\end{equation}
	\end{claim} 
\noindent
\noindent
\noindent
In fact, since \( \Omega \) is a domain with a continuous boundary, the space \( C^{\infty}_{c}(\Omega) \) is dense in \( W^{s, G}_{0}(\Omega) \) (see \cite{baalal2024density}). So, it is sufficient to verify \eqref{3comequ10} for any non-negative function \( \psi \in C^{\infty}_{c}(\Omega) \). To this end, for \( t \in (0, 1) \), we define
\[
\psi_{t} := \min \left\lbrace w_{0} + t \psi, \, \overline{u} \right\rbrace, \quad \text{and} \quad w_{t} := \left(w_{0} + t \psi - \overline{u}\right)^{+}.
\]
It is readily seen that \( w_{t} \leq t \psi \). Moreover, \( w_{t} \in W^{s, G}_{0}(\Omega) \cap \mathcal{K} \) for sufficiently small values of \( t \). Indeed, since \( w_{0} \in \mathcal{K} \), it follows that \( w_{t} = 0 \) in \( \mathbb{R}^{N} \setminus \operatorname{supp}(\psi) \). Also, since the map  $  \mathbb{R} \ni \tau \to \tau^{+}, $ is a Lipschitz function, we remark
 that
	\begin{equation}\label{3equ3}
	w_{t} \in W^{s, G}_{\text{loc}}(\Omega).
	\end{equation}
\noindent
Hence, by Lemma~\ref{lemma2}, and considering a compact set \( K \subset \Omega \) such that \( \operatorname{supp}(\psi) \Subset K \), we deduce that
\begin{align*}
& \min\left\lbrace \left\| w_{t}\right\|^{p^{-}} _{W^{s, G}_{0}(\Omega)},  \left\| w_{t}\right\|^{p^{+ }} _{W^{s, G}_{0}(\Omega)}\right\rbrace \leq  \int_{\mathbb{R}^{N}}\int_{\mathbb{R}^{N}} G\left(  \dfrac{\left| w_{t}(x) - w_{t}(y) \right|}{\left| x - y \right|^{s}}\right) \dfrac{dx dy}{\left| x - y\right| ^{N}}  \\[4pt]
&= \int_{K} \int_{K}  G\left(  \dfrac{\left| w_{t}(x) - w_{t}(y) \right|}{\left| x - y \right|^{s}}\right) \dfrac{dx dy}{\left| x - y\right| ^{N}}
+  2 \int_{\operatorname{supp}(\psi)} \left( \int_{\mathbb{R}^{N} \setminus K} G\left(  \dfrac{\left| w_{t}(x) \right|}{\left| x - y \right|^{s}}\right) \dfrac{dy}{\left| x - y\right| ^{N}} \right) dx.
\end{align*}
On the other hand, from \eqref{3equ3}, we observe that the first term on the second line is finite. Moreover, we have
\begin{align*}
&\int_{\operatorname{supp}(\psi)} \left( \int_{\mathbb{R}^{N} \setminus K} G\left(  \dfrac{\left| w_{t}(x) \right|}{\left| x - y \right|^{s}}\right) \dfrac{dy}{\left| x - y\right| ^{N}} \right) dx \leq \\[4pt]
&\sup_{x \in \operatorname{supp}(\psi)} \int_{\mathbb{R}^{N} \setminus K}\left[  \max\left\lbrace \dfrac{1}{\left| x - y \right|^{sp^{-}}} , \dfrac{1}{\left| x - y \right|^{sp^{+}}} \right\rbrace  \dfrac{dy}{\left| x - y \right|^{N}} \right]
\max\left\lbrace \left\| w_{t}\right\|^{p^{-}} _{L^{G}(\operatorname{supp}(\psi))},  \left\| w_{t}\right\|^{p^{+ }} _{L^{G}(\operatorname{supp}(\psi))}\right\rbrace < \infty.
\end{align*}
Additionally, we note that
\begin{equation}\label{equ32}
\psi_{t} - w_{0} = t \psi - w_{t}.
\end{equation}
Hence, we conclude that \( \psi_{t} \in w_{0} + \left( W^{s, G}_{0}(\Omega) \cap L_{c}^{\infty}(\Omega) \right) \) and, in particular, \( \psi_{t} \in \mathcal{K} \). Therefore, from \eqref{3comequ8}, we obtain
\begin{equation*}
\begin{aligned}
\int_{\mathbb{R}^{N}} \int_{\mathbb{R}^{N}} & g\left( \frac{\left| w_{0}(x) - w_{0}(y) \right| }{\left| x - y \right|^{s}} \right) 
\frac{w_{0}(x) - w_{0}(y)}{\left| w_{0}(x) - w_{0}(y) \right|} \cdot 
\frac{[(\psi_{t} - w_{0})(x) - (\psi_{t} - w_{0})(y)]}{\left| x - y \right|^{N + s}} \, dx \, dy \\[4pt]
&\geq \int_{\Omega} \left[ f(x) (w_{0} + \epsilon)^{-\alpha} + k(x) (w_{0} + \epsilon)^{\beta} \right] (\psi_{t} - w_{0}) \, dx.
\end{aligned}
\end{equation*}
Combining this with \eqref{equ32}, we infer that
\begin{equation*}
\begin{aligned}
\int_{\mathbb{R}^{N}} \int_{\mathbb{R}^{N}} & g\left( \frac{\left| w_{0}(x) - w_{0}(y) \right|}{\left| x - y \right|^{s}} \right) 
\frac{w_{0}(x) - w_{0}(y)}{\left| w_{0}(x) - w_{0}(y) \right|} \cdot 
\frac{\psi(x) - \psi(y)}{\left| x - y \right|^{N + s}} \, dx \, dy \\[6pt]
&\quad - \int_{\Omega} \left[ f(x) (w_{0} + \epsilon)^{-\alpha} + k(x) (w_{0} + \epsilon)^{\beta} \right] \psi(x) \, dx \\[6pt]
&\geq \frac{1}{t} \left( 
\int_{\mathbb{R}^{N}} \int_{\mathbb{R}^{N}} 
g\left( \frac{\left| w_{0}(x) - w_{0}(y) \right|}{\left| x - y \right|^{s}} \right) 
\frac{w_{0}(x) - w_{0}(y)}{\left| w_{0}(x) - w_{0}(y) \right|} \cdot 
\frac{w_{t}(x) - w_{t}(y)}{\left| x - y \right|^{N + s}} \, dx \, dy \right. \\[6pt]
&\qquad \left. - \int_{\Omega} \left[ f(x) (w_{0} + \epsilon)^{-\alpha} + k(x) (w_{0} + \epsilon)^{\beta} \right] w_{t}(x) \, dx \right).
\end{aligned}
\end{equation*}
Since $\overline{u}$ is a weak super-solution to problem \eqref{GP}, we obtain
{\footnotesize 	\begin{equation}\label{equ5} 
\begin{gathered}
\begin{aligned}
&\int_{\mathbb{R}^{N}} \int_{\mathbb{R}^{N}}  g\left( \frac{\left| w_{0}(x) - w_{0}(y) \right|}{\left| x - y \right|^{s}} \right) 
\frac{w_{0}(x) - w_{0}(y)}{\left| w_{0}(x) - w_{0}(y) \right|} \cdot 
\frac{\psi(x) - \psi(y)}{\left| x - y \right|^{N + s}} \, dx \, dy \\[6pt]
&-\int_{\Omega} \left[ f(x) (w_{0} + \epsilon)^{-\alpha} + k(x) (w_{0} + \epsilon)^{\beta} \right] \psi(x) \, dx \\[6pt]
&\geq   \underbrace{ \dfrac{1}{t} 
	\iint_{\mathbb{R}^{2N}} 
	\left[ g\left( \frac{\left| w_{0}(x) - w_{0}(y) \right|}{\left| x - y \right|^{s}} \right) 
	\frac{w_{0}(x) - w_{0}(y)}{\left| w_{0}(x) - w_{0}(y) \right|} - g\left( \frac{\left| \overline{u}(x) - \overline{u}(y) \right|}{\left| x - y \right|^{s}} \right) 
	\frac{\overline{u}(x) - \overline{u}(y)}{\left|\overline{u}(x) - \overline{u}(y) \right|} \right] \cdot 
	\frac{w_{t}(x) - w_{t}(y)}{\left| x - y \right|^{N + s}} \, dx \, dy  }_{\textbf{I}_{1}} \\[4pt]
& + \underbrace{ \dfrac{1}{t} \left( {\displaystyle\int_{\Omega}}\left( f(x)\,\left(\overline{u}^{-\alpha} - (w_{0} + \epsilon)^{-\alpha} \right) + k(x) \left(\overline{u}^{\beta} - (w_{0} + \epsilon)^{\beta} \right) \right) \,w_{t}\,dx\right) }_{\textbf{I}_{2}} .
\end{aligned}
\end{gathered}
\end{equation}}

\noindent
\textbf{Estimate of $\mathbf{I}_{1}$.} We begin by defining the sets
\[
\omega_{1} = \left\lbrace x \in \Omega \mid \overline{u}(x) \leq (w_{0} + t \psi)(x) \right\rbrace \quad \text{and} \quad 
\omega_{2} = \left\lbrace x \in \Omega \mid \overline{u}(x) \geq (w_{0} + t \psi)(x) \right\rbrace.
\]
By exploiting the monotonicity of the mapping
\[
\tau \mapsto g(\left| \tau - \tau_{0}\right|) \dfrac{\tau - \tau_{0}}{\left| \tau - \tau_{0}\right|},
\]
we deduce that
	\begin{align*}
	\int_{\mathbb{R}^{N}} \int_{\mathbb{R}^{N}} &
	g\left( \frac{\left| \psi_{t}(x) - \psi_{t}(y) \right|}{\left| x - y \right|^{s}} \right) 
	\frac{\psi_{t}(x) - \psi_{t}(y)}{\left| \psi_{t}(x) - \psi_{t}(y) \right|} \cdot 
	\frac{w_{t}(x) - w_{t}(y)}{\left| x - y \right|^{N + s}} \, dx \, dy\\[4pt]
	&  =  \displaystyle\iint_{\omega_{1}\times \omega_{1}}  
	g\left( \frac{\left| \psi_{t}(x) - \psi_{t}(y) \right|}{\left| x - y \right|^{s}} \right) 
	\frac{\psi_{t}(x) - \psi_{t}(y)}{\left| \psi_{t}(x) - \psi_{t}(y) \right|} \cdot 
	\frac{w_{t}(x) - w_{t}(y)}{\left| x - y \right|^{N + s}} \, dx \, dy\\[4pt]
	& \quad+ \displaystyle\iint_{\omega_{1} \times \omega_{2}}  
	g\left( \frac{\left| \psi_{t}(x) - \psi_{t}(y) \right|}{\left| x - y \right|^{s}} \right) 
	\frac{\psi_{t}(x) - \psi_{t}(y)}{\left| \psi_{t}(x) - \psi_{t}(y) \right|} \cdot 
	\frac{w_{t}(x) - w_{t}(y)}{\left| x - y \right|^{N + s}} \, dx \, dy\\[4pt]
	&\quad +  \displaystyle\iint_{\omega_{2} \times \omega_{1}}  
	g\left( \frac{\left| \psi_{t}(x) - \psi_{t}(y) \right|}{\left| x - y \right|^{s}} \right) 
	\frac{\psi_{t}(x) - \psi_{t}(y)}{\left| \psi_{t}(x) - \psi_{t}(y) \right|} \cdot 
	\frac{w_{t}(x) - w_{t}(y)}{\left| x - y \right|^{N + s}} \, dx \, dy \\[4pt]
	&\geq \displaystyle\iint_{\omega_{1}\times \omega_{1}}  
	g\left( \frac{\left| \overline{u}(x) - \overline{u}(y) \right|}{\left| x - y \right|^{s}} \right) 
	\frac{\overline{u}(x) - \overline{u}(y)}{\left| \overline{u}(x) - \overline{u}(y) \right|} \cdot 
	\frac{w_{t}(x) - w_{t}(y)}{\left| x - y \right|^{N + s}} \, dx \, dy\\[4pt]
	&\quad + \displaystyle\iint_{\omega_{1} \times \omega_{2}}  
	g\left( \frac{\left| \overline{u}(x) - \overline{u}(y) \right|}{\left| x - y \right|^{s}} \right) 
	\frac{\overline{u}(x) - \overline{u}(y)}{\left|\overline{u}(x) - \overline{u}(y) \right|} \cdot 
	\frac{w_{t}(x) - w_{t}(y)}{\left| x - y \right|^{N + s}} \, dx \, dy\\[4pt]
	&\quad +  \displaystyle\iint_{\omega_{2} \times \omega_{1}}  
	g\left( \frac{\left| \overline{u}(x) - \overline{u}(y) \right|}{\left| x - y \right|^{s}} \right) 
	\frac{\overline{u}(x) - \overline{u}(y)}{\left| \overline{u}(x) - \overline{u}(y) \right|} \cdot 
	\frac{w_{t}(x) - w_{t}(y)}{\left| x - y \right|^{N + s}} \, dx \, dy \\[4pt]
	& = \int_{\mathbb{R}^{N}} \int_{\mathbb{R}^{N}} 
	g\left( \frac{\left| \overline{u}(x) - \overline{u}(y) \right|}{\left| x - y \right|^{s}} \right) 
	\frac{\overline{u}(x) -\overline{u}(y)}{\left| \overline{u}(x) - \overline{u}(y) \right|} \cdot 
	\frac{w_{t}(x) - w_{t}(y)}{\left| x - y \right|^{N + s}} \, dx \, dy.
	\end{align*}
Now, using the inequality (see \cite[Lemma 2.3]{bal2024singular}), there exists a constant \( C > 0 \), depending on \( G \), such that
\begin{align}\label{3equ6}
\left( g\left( \left| b \right| \right) \dfrac{b}{\left| b \right|} - g\left( \left| a \right| \right) \dfrac{a}{\left| a \right|} \right)(b - a) 
\geq C\, G\left( \left| b - a \right| \right) \geq 0, \quad \forall\, a, b \in \mathbb{R},
\end{align}
and hence we obtain, by defining 
\[
\Gamma_{\mathrm{supp}(w_{t})} := \mathbb{R}^{2N} \setminus \left( \mathrm{supp}(w_{t}) \right)^{c} \times \left( \mathrm{supp}(w_{t}) \right)^{c},
\]
the following estimate:
{\small
	\begin{align*}
	&\iint_{\Gamma_{\mathrm{supp}(w_{t})}} 
	\left[
	g\left( \frac{\left| w_{0}(x) - w_{0}(y) \right|}{\left| x - y \right|^{s}} \right) 
	\frac{w_{0}(x) - w_{0}(y)}{\left| w_{0}(x) - w_{0}(y) \right|} - 
	g\left( \frac{\left| \psi_{t}(x) - \psi_{t}(y) \right|}{\left| x - y \right|^{s}} \right) 
	\frac{\psi_{t}(x) - \psi_{t}(y)}{\left| \psi_{t}(x) - \psi_{t}(y) \right|}
	\right] \cdot 
	\frac{w_{0}(x) - w_{0}(y)}{\left| x - y \right|^{N + s}} \, dx \, dy \\[4pt]
	&\geq 
	\iint_{\Gamma_{\mathrm{supp}(w_{t})}} 
	\left[
	g\left( \frac{\left| w_{0}(x) - w_{0}(y) \right|}{\left| x - y \right|^{s}} \right) 
	\frac{w_{0}(x) - w_{0}(y)}{\left| w_{0}(x) - w_{0}(y) \right|} - 
	g\left( \frac{\left| \psi_{t}(x) - \psi_{t}(y) \right|}{\left| x - y \right|^{s}} \right) 
	\frac{\psi_{t}(x) - \psi_{t}(y)}{\left| \psi_{t}(x) - \psi_{t}(y) \right|}
	\right] \cdot 
	\frac{\psi_{t}(x) - \psi_{t}(y)}{\left| x - y \right|^{N + s}} \, dx dy.
	\end{align*}
}

\noindent
By combining the above estimates with~\eqref{equ32}, and defining
\[
\Gamma^{1}_{\mathrm{supp}(w_{t})} := \mathrm{supp}(w_{t}) \times \mathrm{supp}(w_{t}), 
\quad \text{and} \quad 
\Gamma^{2}_{\mathrm{supp}(w_{t})} := \mathrm{supp}(w_{t}) \times \mathrm{supp}(w_{t})^{c},
\]
we deduce that
{\scriptsize 
	\begin{equation}\label{equ33}
	\begin{aligned}
	\mathbf{I}_{1} & \geq \iint_{\Gamma_{\mathrm{supp}(w_{t})}} 
	\left[ g\left( \frac{\left| w_{0}(x) - w_{0}(y) \right|}{\left| x - y \right|^{s}} \right) 
	\frac{w_{0}(x) - w_{0}(y)}{\left| w_{0}(x) - w_{0}(y) \right|} 
	- g\left( \frac{\left| \psi_{t}(x) - \psi_{t}(y) \right|}{\left| x - y \right|^{s}} \right) 
	\frac{\psi_{t}(x) - \psi_{t}(y)}{\left| \psi_{t}(x) - \psi_{t}(y) \right|} \right]  \cdot \frac{\psi(x) - \psi(y)}{\left| x - y \right|^{N + s}} \, dx \, dy \\[4pt]
	&= \iint_{\Gamma^{1}_{\mathrm{supp}(w_{t})}} 
	\left[ g\left( \frac{\left| w_{0}(x) - w_{0}(y) \right|}{\left| x - y \right|^{s}} \right) 
	\frac{w_{0}(x) - w_{0}(y)}{\left| w_{0}(x) - w_{0}(y) \right|} 
	- g\left( \frac{\left| \overline{u}(x) - \overline{u}(y) \right|}{\left| x - y \right|^{s}} \right) 
	\frac{\overline{u}(x) - \overline{u}(y)}{\left| \overline{u}(x) - \overline{u}(y) \right|} \right] \cdot \frac{\psi(x) - \psi(y)}{\left| x - y \right|^{N + s}} \, dx \, dy \\[4pt]
	&+ 2 \iint_{\Gamma^{2}_{\mathrm{supp}(w_{t})}} 
	\left[ g\left( \frac{\left| w_{0}(x) - w_{0}(y) \right|}{\left| x - y \right|^{s}} \right) 
	\frac{w_{0}(x) - w_{0}(y)}{\left| w_{0}(x) - w_{0}(y) \right|} 
	- g\left( \frac{\left| \psi_{t}(x) - \psi_{t}(y) \right|}{\left| x - y \right|^{s}} \right) 
	\frac{\psi_{t}(x) - \psi_{t}(y)}{\left| \psi_{t}(x) - \psi_{t}(y) \right|} \right] \cdot \frac{\psi(x) - \psi(y)}{\left| x - y \right|^{N + s}} \, dx \, dy \\[4pt]
	&\quad \longrightarrow 0 \quad \text{as } t \to 0^{+}.
	\end{aligned}
	\end{equation}
}

\noindent
Indeed, since \( \text{meas}(\mathrm{supp}(w_{t})) \to 0 \) as \( t \to 0^{+} \), it is straightforward to verify that
{\small
	\begin{equation*}
	\begin{aligned}
	\iint_{\Gamma^{1}_{\mathrm{supp}(w_{t})}} & 
	\left[ g\left( \frac{\left| w_{0}(x) - w_{0}(y) \right|}{\left| x - y \right|^{s}} \right) 
	\frac{w_{0}(x) - w_{0}(y)}{\left| w_{0}(x) - w_{0}(y) \right|} - 
	g\left( \frac{\left| \overline{u}(x) - \overline{u}(y) \right|}{\left| x - y \right|^{s}} \right) 
	\frac{\overline{u}(x) - \overline{u}(y)}{\left| \overline{u}(x) - \overline{u}(y) \right|} \right] \cdot 
	\frac{\psi(x) - \psi(y)}{\left| x - y \right|^{N + s}} \, dx \, dy \\[4pt]
	&\longrightarrow 0 \quad \text{as } t \to 0^{+}.
	\end{aligned}
	\end{equation*}
}

\noindent
Now, by employing the following inequality (see \cite[Lemma 3.10]{bal2024singular}), valid for some constant \( C > 0 \),
\[
\left| g(|a|)\, \frac{a}{|a|} - g(|b|)\, \frac{b}{|b|} \right| \leq C\, g(|a| + |b|), \quad \forall\, a, b \in \mathbb{R},
\]
and applying H\"{o}lder's inequality (see Lemma~\ref{lemma3}) with respect to the measure \( \frac{dx\,dy}{|x - y|^N} \), we deduce that
{\scriptsize
	\begin{equation}\label{3equ36}
	\begin{aligned}
	& \left| \iint_{\Gamma^{2}_{\mathrm{supp}(w_{t})}} 
	\left[ g\left( \frac{\left| w_{0}(x) - w_{0}(y) \right|}{\left| x - y \right|^{s}} \right) 
	\frac{w_{0}(x) - w_{0}(y)}{\left| w_{0}(x) - w_{0}(y) \right|} 
	- g\left( \frac{\left| \psi_{t}(x) - \psi_{t}(y) \right|}{\left| x - y \right|^{s}} \right) 
	\frac{\psi_{t}(x) - \psi_{t}(y)}{\left| \psi_{t}(x) - \psi_{t}(y) \right|} \right] 
	\cdot \frac{\psi(x) - \psi(y)}{\left| x - y \right|^{N + s}} \, dx \, dy \right| \\[4pt]
	& \leq \iint_{\Gamma^{2}_{\mathrm{supp}(w_{t})}} 
	\left| 
	g\left( \frac{\left| w_{0}(x) - w_{0}(y) \right|}{\left| x - y \right|^{s}} \right) 
	\frac{w_{0}(x) - w_{0}(y)}{\left| w_{0}(x) - w_{0}(y) \right|} 
	- g\left( \frac{\left| \psi_{t}(x) - \psi_{t}(y) \right|}{\left| x - y \right|^{s}} \right) 
	\frac{\psi_{t}(x) - \psi_{t}(y)}{\left| \psi_{t}(x) - \psi_{t}(y) \right|} 
	\right| \cdot 
	\frac{\left| \psi(x) - \psi(y) \right|}{\left| x - y \right|^{N + s}} \, dx \, dy \\[4pt]
	& \leq C \iint_{\Gamma^{2}_{\mathrm{supp}(w_{t})}} 
	g\left( \frac{\left| w_{0}(x) - w_{0}(y) \right| + \left| \psi_{t}(x) - \psi_{t}(y) \right|}{\left| x - y \right|^{s}} \right) 
	\cdot \frac{\left| \psi(x) - \psi(y) \right|}{\left| x - y \right|^{N + s}} \, dx \, dy \\[4pt]
	& \leq C \left\|g\left( \frac{\left| w_{0}(x) - w_{0}(y) \right| + \left| \psi_{t}(x) - \psi_{t}(y) \right|}{\left| x - y \right|^{s}} \right) \right\| _{L^{\overline{G}}(\Gamma^{2}_{\mathrm{supp}(w_{t})}, \frac{dx\,dy}{|x - y|^N})} \left\|\frac{ \psi(x) - \psi(y)}{\left| x - y \right|^{s}}\right\| _{L^{G}(\Gamma^{2}_{\mathrm{supp}(w_{t})}, \frac{dx\,dy}{|x - y|^N})}.
	\end{aligned}
	\end{equation}
}

\noindent
Now, we distinguish two cases:\\[4pt]
\textbf{Case 1:} If
\[
\left\|g\left( \frac{\left| w_{0}(x) - w_{0}(y) \right| + \left| \psi_{t}(x) - \psi_{t}(y) \right|}{\left| x - y \right|^{s}} \right) \right\| _{L^{\overline{G}}(\Gamma^{2}_{\mathrm{supp}(w_{t})}, \frac{dx\,dy}{|x - y|^N})} \leq 1,
\]
then we obtain
{\footnotesize 
	\begin{equation*}
	\begin{aligned}
	& \left| \iint_{\Gamma^{2}_{\mathrm{supp}(w_{t})}} 
	\left[ g\left( \frac{\left| w_{0}(x) - w_{0}(y) \right|}{\left| x - y \right|^{s}} \right) 
	\frac{w_{0}(x) - w_{0}(y)}{\left| w_{0}(x) - w_{0}(y) \right|} 
	- g\left( \frac{\left| \psi_{t}(x) - \psi_{t}(y) \right|}{\left| x - y \right|^{s}} \right) 
	\frac{\psi_{t}(x) - \psi_{t}(y)}{\left| \psi_{t}(x) - \psi_{t}(y) \right|} \right] 
	\cdot \frac{\psi(x) - \psi(y)}{\left| x - y \right|^{N + s}} \, dx \, dy \right| \\[4pt]
	& \leq C \left\| \frac{ \psi(x) - \psi(y)}{\left| x - y \right|^{s}} \right\|_{L^{G}(\Gamma^{2}_{\mathrm{supp}(w_{t})}, \frac{dx\,dy}{|x - y|^N})} \longrightarrow 0 \text{ as } t \to 0^{+}.
	\end{aligned}
	\end{equation*}
}

\noindent
\textbf{Case 2:} Otherwise, since \( \mathrm{supp}(w_{t}) \Subset \mathrm{supp}(\psi) \) and according to Definition~\ref{definition1}, there exists a constant \( C_{\mathrm{supp}(\psi)} > 0 \), independent of \( t \), such that \( \overline{u} > C_{\mathrm{supp}(\psi)} \) in \( \mathrm{supp}(w_{t}) \). Hence, by Lemma~\ref{lemma2}, estimate~\eqref{3equ03}, and \cite[Lemma 3.11]{bal2024singular}, we obtain

{\small
	\begin{equation*}\label{3equ37}
	\begin{aligned}
	& \left\| g\left( \frac{ \left| w_{0}(x) - w_{0}(y) \right| + \left| \psi_{t}(x) - \psi_{t}(y) \right| }{ \left| x - y \right|^{s} } \right) \right\|_{L^{\overline{G}}(\Gamma^{2}_{\mathrm{supp}(w_{t})}, \frac{dx\,dy}{|x - y|^N})} \\[4pt]
	& \leq \left( \iint_{\Gamma^{2}_{\mathrm{supp}(w_{t})}} 
	\overline{G}\left( g\left( \frac{ \left| w_{0}(x) - w_{0}(y) \right| + \left| \psi_{t}(x) - \psi_{t}(y) \right| }{ \left| x - y \right|^{s} } \right) \right) 
	\frac{dx \, dy}{\left| x - y \right|^{N}} \right)^{\frac{1}{p^{-}}} \\[4pt]
	& \leq (p^{+} - 1)^{\frac{1}{p^{-}}}\left( \iint_{\Gamma^{2}_{\mathrm{supp}(w_{t})}} 
	G\left( \frac{ \left| w_{0}(x) - w_{0}(y) \right| + \left| \psi_{t}(x) - \psi_{t}(y) \right| }{ \left| x - y \right|^{s} } \right) 
	\frac{dx \, dy}{\left| x - y \right|^{N}} \right)^{\frac{1}{p^{-}}} \\[4pt]
	& \leq C \left( 
	\iint_{\Gamma^{2}_{\mathrm{supp}(w_{t})}} 
	G\left( \frac{ \left| w_{0}(x) - w_{0}(y) \right| }{ \left| x - y \right|^{s} } \right) 
	\frac{dx \, dy}{\left| x - y \right|^{N}} + \iint_{\Gamma^{2}_{\mathrm{supp}(w_{t})}} 
	G\left( \frac{ \left| \psi(x) - \psi(y) \right| }{ \left| x - y \right|^{s} } \right) 
	\frac{dx \, dy}{\left| x - y \right|^{N}}  \right. \\[4pt]
	& \qquad+\left. 
	\iint_{\Gamma^{2}_{\mathrm{supp}(w_{t})}} 
	G\left( \frac{ \left| \Psi(\overline{u}(x)) - \Psi(\overline{u}(y)) \right| }{ \left| x - y \right|^{s} } \right) 
	\frac{dx \, dy}{\left| x - y \right|^{N}} 
	\right)^{\frac{1}{p^{-}}} < \infty,
	\end{aligned}
	\end{equation*}
}

\noindent
where \( C \) depends on \( p^{-} \), \( p^{+} \), \( \Psi \), and \( C_{\mathrm{supp}(\psi)} \), with \( \Psi \in \mathbf{FC} \). Hence, by applying Lemma~\ref{lemma2}, we deduce that
{\footnotesize 
	\begin{equation*}
	\begin{aligned}
	& \left| \iint_{\Gamma^{2}_{\mathrm{supp}(w_{t})}} 
	\left[ g\left( \frac{\left| w_{0}(x) - w_{0}(y) \right|}{\left| x - y \right|^{s}} \right) 
	\frac{w_{0}(x) - w_{0}(y)}{\left| w_{0}(x) - w_{0}(y) \right|} 
	- g\left( \frac{\left| \psi_{t}(x) - \psi_{t}(y) \right|}{\left| x - y \right|^{s}} \right) 
	\frac{\psi_{t}(x) - \psi_{t}(y)}{\left| \psi_{t}(x) - \psi_{t}(y) \right|} \right] 
	\cdot \frac{\psi(x) - \psi(y)}{\left| x - y \right|^{N + s}} \, dx \, dy \right| \\[4pt]
	& \leq C \left\| \frac{ \psi(x) - \psi(y)}{\left| x - y \right|^{s}} \right\|_{L^{G}(\Gamma^{2}_{\mathrm{supp}(w_{t})}, \frac{dx\,dy}{|x - y|^N})} \longrightarrow 0 \quad \text{as } t \to 0^{+}.
	\end{aligned}
	\end{equation*}
}

\noindent
\textbf{Estimate of $\textbf{I}_3$.} Since $w_t \leq t \psi$, we observe that
\begin{equation}\label{equ30}
\begin{aligned}
\textbf{I}_3 &\geq - \int_{\operatorname{supp}(w_t)} \left( f(x) \left| \overline{u}^{-\alpha} - (w_0 + \epsilon)^{-\alpha} \right| + k(x) \left| \overline{u}^{\beta} - (w_0 + \epsilon)^{\beta} \right| \right) \psi\, dx \\[4pt]
&\quad \longrightarrow 0 \quad \text{as } t \to 0^{+}.
\end{aligned}
\end{equation}
Combining \eqref{equ33} and \eqref{equ30}, and passing to the limit in \eqref{equ5} as $t \to 0^{+}$, we obtain
\begin{equation}\label{equ6}
\begin{aligned}
\int_{\mathbb{R}^{N}} \int_{\mathbb{R}^{N}}& g\left( \frac{\left| w_{0}(x) - w_{0}(y)\right| }{\vert x - y \vert^{s}} \right) \dfrac{w_{0}(x) - w_{0}(y)}{\left| w_{0}(x) - w_{0}(y) \right|} \cdot \frac{\psi(x) - \psi(y)}{\vert x - y \vert^{N + s}} \, dx \, dy \\[4pt]
&\quad \geq \int_{\Omega} \left[ f(x) (w_{0} + \epsilon)^{-\alpha} + k(x) (w_{0} + \epsilon)^{\beta} \right] \psi\, dx,
\end{aligned}
\end{equation}
for all \( \psi \in C_c^{\infty}(\Omega) \) with \( \psi \geq 0 \) almost everywhere in \( \Omega \). By employing classical density arguments, we conclude that inequality \eqref{equ6} remains valid for any \( \psi \in W_0^{s, G}(\Omega) \cap L_{c}^{\infty}(\Omega) \) satisfying \( \psi \geq 0 \) a.e. in \( \Omega \). Indeed, let \( (\psi_{\tau,n}) \subset C_c^{\infty}(\Omega) \) be a sequence such that
\[
\psi_{\tau,n} \xrightarrow[\tau \to 0]{} \psi_n \xrightarrow[n \to \infty]{} \psi \quad \text{in } W_0^{s, G}(\Omega),
\]
where \( \operatorname{supp}(\psi_n) \Subset \Omega \) and \( 0 \leq \psi_n \leq \psi \) for all \( n \in \mathbb{N} \). Passing first to the limit as \( \tau \to 0 \), and then as \( n \to \infty \), yields the desired result.
\begin{claim}
	For every $ \epsilon > 0 $, we have $ \underline{u} \leq w_0 + \epsilon $.
\end{claim}

\noindent
Let \( \epsilon > 0 \), \( m > 0 \), and fix \( \mathfrak{k} > 0 \). We define the truncation function \( \mathbf{T}_{\mathfrak{k}} \) by
$$ \mathbf{T}_\mathfrak{k}(s) := \min\{s, \mathfrak{k}\} \text{ if } s \geq 0, \quad \mathbf{T}_\mathfrak{k}(s) := -\min\{-s, \mathfrak{k}\} \text{ if } s < 0. $$
We begin by considering the auxiliary functions
{\small \[
\Psi_{m} := \frac{\mathbf{T}_{\mathfrak{k}}\left( \left[(\underline{u} + m)^{p^{-}} - (w_0 + m + \epsilon)^{p^{-}}\right]^+ \right)}{(\underline{u} + m)^{p^{-} - 1}}, \qquad
\Phi_{m} := \frac{\mathbf{T}_{\mathfrak{k}}\left( \left[(w_0 + m + \epsilon)^{p^{-}} - (\underline{u} + m)^{p^{-}}\right]^- \right)}{(w_0 + m + \epsilon)^{p^{-} - 1}},
\]}

\noindent
which are well defined and belong to \( W_0^{s,G}(\Omega) \cap L^{\infty}(\Omega) \). For brevity, we provide a detailed proof only for \( \Psi_{m} \), since the analysis for \( \Phi_{m} \) is analogous.  Assuming that \( w_0 \geq 0 \) almost everywhere in \( \Omega \), it follows that
\[
\operatorname{supp}(\Psi_{m}) \subset \mathcal{S}_{\epsilon}, \quad \text{where } \mathcal{S}_{\epsilon} := \operatorname{supp}\left((\underline{u} - \epsilon)^+\right).
\]
We introduce the following measurable subsets of \( \Omega \):
\[
\begin{aligned}
\Omega_1 &:= \left\{ x \in \Omega : (\underline{u} + m)^{p^{-}}(x) - (w_0 + m + \epsilon)^{p^{-}}(x) \leq 0 \right\}, \\[4pt]
\Omega_2 &:= \left\{ x \in \Omega : 0 < (\underline{u} + m)^{p^{-}}(x) - (w_0 + m + \epsilon)^{p^{-}}(x) < \mathfrak{k} \right\}, \\[4pt]
\Omega_3 &:= \left\{ x \in \Omega : \mathfrak{k} \leq  (\underline{u} + m)^{p^{-}}(x) - (w_0 + m + \epsilon)^{p^{-}}(x)  \right\}.
\end{aligned}
\]
Our aim is to derive the following estimate, which holds for every pair \( x, y \in \mathbb{R}^N \):
\begin{equation}\label{equ4}
G\left( \dfrac{\left| \Psi_{m}(x) - \Psi_{m}(y) \right| }{ \left| x - y \right|^{s} } \right) 
\leq C \left[ 
G\left( \dfrac{ \left| \underline{u}(x) - \underline{u}(y) \right| }{ \left| x - y \right|^{s} } \right) 
+ 
G\left( \dfrac{ \left| w_0(x) - w_0(y) \right| }{ \left| x - y \right|^{s} } \right) 
\right],
\end{equation}
for some constant \( C > 0 \). By symmetry, it suffices to consider the following cases: \\[6pt]
\textbf{(1)} If \(x, y \in \Omega_1\), then \(\Psi_m(x) = \Psi_m(y) = 0\).\\[6pt]
\textbf{(2)} Let \(y \in \Omega_1\) and \(x \in \Omega_2\). By the Lagrange Mean Value Theorem, we have
\begin{equation*}
\begin{aligned}
\left| \Psi_m(x) - \Psi_m(y) \right| 
&= \frac{(\underline{u} + m)^{p^{-}}(x) - (w_0 + m + \epsilon)^{p^{-}}(x)}{(\underline{u} + m)^{p^{-}-1}(x)} \\[4pt]
&\leq p^{-} \cdot \frac{\max\left\{ (\underline{u} + m)^{p^{-}-1}(x),\, (w_0 + m + \epsilon)^{p^{-}-1}(x) \right\}}{(\underline{u} + m)^{p^{-}-1}(x)} \cdot \left(  \underline{u}(x) - \epsilon - w_0(x) \right) \\[4pt]
&\leq  p^{-} \left[ \left| \underline{u}(x) - \underline{u}(y) \right| + \left| w_0(x) - w_0(y) \right| \right],
\end{aligned}
\end{equation*}
where we used the fact that \( \underline{u}(y) - \epsilon - w_0(y) \leq 0 \) by construction.\\[6pt]
\textbf{(3)} Let \(y \in \Omega_1\) and \(x \in \Omega_3\). Then
\begin{equation*}
\begin{aligned}
\left| \Psi_m(x) - \Psi_m(y) \right| 
&= \frac{\mathfrak{k}}{(\underline{u} + m)^{p^{-}-1}(x)} \\[4pt]
&\leq \frac{(\underline{u} + m)^{p^{-}}(x) - (w_0 + m + \epsilon)^{p^{-}}(x)}{(\underline{u} + m)^{p^{-}-1}(x)} \\[4pt]
&\leq p^{-} \left[ \left| \underline{u}(x) - \underline{u}(y) \right| + \left| w_0(x) - w_0(y) \right| \right].
\end{aligned}
\end{equation*}
\textbf{(4)} If \(x, y \in \Omega_2\) with \(\underline{u}(x) > \underline{u}(y)\), then
\begin{equation*}
\begin{aligned}
&\left| \Psi_m(x) - \Psi_m(y) \right| 
= \left| \underline{u}(x) - \frac{(w_0 + m + \epsilon)^{p^{-}}(x)}{(\underline{u} + m)^{p^{-} - 1}(x)} - \underline{u}(y) + \frac{(w_0 + m + \epsilon)^{p^{-}}(y)}{(\underline{u} + m)^{p^{-} - 1}(y)} \right| \\[4pt]
&\leq \left| \underline{u}(x) - \underline{u}(y) \right| 
+ \left| \frac{(w_0 + m + \epsilon)^{p^{-}}(x)}{(\underline{u} + m)^{p^{-} - 1}(x)} - \frac{(w_0 + m + \epsilon)^{p^{-}}(y)}{(\underline{u} + m)^{p^{-} - 1}(y)} \right| \\[4pt]
&\leq \left| \underline{u}(x) - \underline{u}(y) \right| 
+ \frac{\left| (w_0 + m + \epsilon)^{p^{-}}(x) - (w_0 + m + \epsilon)^{p^{-}}(y) \right|}{(\underline{u} + m)^{p^{-} - 1}(x)} \\[4pt]
&\quad + (w_0 + m + \epsilon)^{p^{-}}(y) 
\left| \frac{1}{(\underline{u} + m)^{p^{-} - 1}(x)} - \frac{1}{(\underline{u} + m)^{p^{-} - 1}(y)} \right| \\[4pt]
&\leq \left| \underline{u}(x) - \underline{u}(y) \right| 
+ p^{-} \cdot \frac{\max\left\{ (w_0 + m + \epsilon)^{p^{-}-1}(x),\, (w_0 + m + \epsilon)^{p^{-}-1}(y) \right\}}{(\underline{u} + m)^{p^{-} - 1}(x)} \left| w_0(x) - w_0(y) \right| \\[4pt]
&\quad + (p^{-} - 1)(w_0 + m + \epsilon)(y) 
\cdot \frac{\max\left\{ (\underline{u} + m)^{p^{-}-2}(x),\, (\underline{u} + m)^{p^{-}-2}(y) \right\}}{(\underline{u} + m)^{p^{-}-1}(x)} \left| \underline{u}(x) - \underline{u}(y) \right| \\[4pt]
&\leq p^{-} \left( \left| \underline{u}(x) - \underline{u}(y) \right| + \left| w_0(x) - w_0(y) \right| \right).
\end{aligned}
\end{equation*}
\textbf{(5)} Let \( y \in \Omega_2 \) and \( x \in \Omega_3 \). Then
\begin{equation*}
\begin{aligned}
\left| \Psi_m(x) - \Psi_m(y) \right| 
&= \left| \frac{\mathfrak{k}}{(\underline{u} + m)^{p^{-}-1}(x)} - \frac{(\underline{u} + m)^{p^{-}}(y) - (w_0 + m + \epsilon)^{p^{-}}(y)}{(\underline{u} + m)^{p^{-}-1}(y)} \right| \\[4pt]
&\leq \left| \frac{\mathfrak{k}}{(\underline{u} + m)^{p^{-}-1}(x)} - \frac{\mathfrak{k}}{(\underline{u} + m)^{p^{-}-1}(y)} \right| \\[4pt]
&\quad + \left| \frac{(\underline{u} + m)^{p^{-}}(x) - (w_0 + m + \epsilon)^{p^{-}}(x)}{(\underline{u} + m)^{p^{-}-1}(y)} - \frac{(\underline{u} + m)^{p^{-}}(y) - (w_0 + m + \epsilon)^{p^{-}}(y)}{(\underline{u} + m)^{p^{-}-1}(y)} \right| \\[4pt]
&\leq (p^{-} - 1) \, (\underline{u} + m)(x) \, \frac{\max\left\{ (\underline{u} + m)^{p^{-} - 2}(x), (\underline{u} + m)^{p^{-} - 2}(y) \right\}}{(\underline{u} + m)^{p^{-} - 1}(y)} \, \left| \underline{u}(x) - \underline{u}(y) \right| \\[4pt]
&\quad + p^{-} \cdot \frac{\max\left\{ (\underline{u} + m)^{p^{-} - 1}(x), (\underline{u} + m)^{p^{-} - 1}(y) \right\}}{(\underline{u} + m)^{p^{-} - 1}(y)} \, \left| \underline{u}(x) - \underline{u}(y) \right| \\[4pt]
&\quad + p^{-} \cdot \frac{\max\left\{ (w_0 + m + \epsilon)^{p^{-} - 1}(x), (w_0 + m + \epsilon)^{p^{-} - 1}(y) \right\}}{(\underline{u} + m)^{p^{-} - 1}(y)} \, \left| w_0(x) - w_0(y) \right| \\[4pt]
&\leq C \left( \left| \underline{u}(x) - \underline{u}(y) \right| + \left| w_0(x) - w_0(y) \right| \right),
\end{aligned}
\end{equation*}
where \( C > 0 \) is a constant depending only on \( p^{-} \).\\[6pt]
\textbf{(6)} Let \(x, y \in \Omega_3\). Then
\begin{equation*}
\begin{aligned}
\left| \Psi_m(x) - \Psi_m(y) \right| 
&= \left|\frac{\mathfrak{k}}{(\underline{u} + m)^{p^{-} - 1}(x)} - \frac{\mathfrak{k}}{(\underline{u} + m)^{p^{-} - 1}(y)}\right|= \mathfrak{k} \frac{\left|(\underline{u} + m)^{p^{-} - 1}(x) - (\underline{u} + m)^{p^{-} - 1}(y)\right|}{(\underline{u} + m)^{p^{-} - 1}(x)\, (\underline{u} + m)^{p^{-} - 1}(y)} \\[4pt]
&\leq \mathfrak{k}(p^{-} - 1) \frac{\max\left\{ (\underline{u} + m)^{p^{-} - 2}(x),\, (\underline{u} + m)^{p^{-} - 2}(y) \right\}}{(\underline{u} + m)^{p^{-} - 1}(x)\, (\underline{u} + m)^{p^{-} - 1}(y)} \left| \underline{u}(x) - \underline{u}(y) \right| \\[4pt]
&\leq C \left( \left| \underline{u}(x) - \underline{u}(y) \right| + \left| w_0(x) - w_0(y) \right| \right),
\end{aligned}
\end{equation*}
where \(C > 0\) is a constant depending only on \(p^{-}\). By the convexity of \(G\), inequality~\eqref{equ4} holds. Thus, integrating both sides and applying Lemma~\ref{lemma2}, and since \(\underline{u}\) is a weak sub-solution of problem~\eqref{GP}, we obtain
\begin{align*}
&\min\left\{ \left\| \Psi_m \right\|_{W^{s,G}_{0}(\Omega)}^{p^{-}},\, \left\| \Psi_m \right\|_{W^{s,G}_{0}(\Omega)}^{p^{+}} \right\}
\leq \iint_{(\mathbb{R}^{N} \times \mathbb{R}^{N}) \setminus (\mathcal{S}_{\epsilon}^{c} \times \mathcal{S}_{\epsilon}^{c})} G\left( \frac{\left| \Psi_m(x) - \Psi_m(y) \right|}{\left| x - y \right|^{s}} \right) \frac{dx\,dy}{\left| x - y \right|^{N}} \\[4pt]
&\leq C \left[ \iint_{(\mathbb{R}^{N} \times \mathbb{R}^{N}) \setminus (\mathcal{S}_{\epsilon}^{c} \times \mathcal{S}_{\epsilon}^{c})} G\left( \frac{\left| \underline{u}(x) - \underline{u}(y) \right|}{\left| x - y \right|^{s}} \right) \frac{dx\,dy}{\left| x - y \right|^{N}}  + \iint_{\mathbb{R}^{N} \times \mathbb{R}^{N}} G\left( \frac{\left| w_0(x) - w_0(y) \right|}{\left| x - y \right|^{s}} \right) \frac{dx\,dy}{\left| x - y \right|^{N}}\right]  \\[4pt]
&\leq C(\epsilon)\left[  \iint_{\mathbb{R}^{N} \times \mathbb{R}^{N}} G\left( \frac{\left| \Psi(\underline{u}(x)) - \Psi(\underline{u}(y)) \right|}{\left| x - y \right|^{s}} \right) \frac{dx\,dy}{\left| x - y \right|^{N}} + \iint_{\mathbb{R}^{N} \times \mathbb{R}^{N}} G\left( \frac{\left| w_0(x) - w_0(y) \right|}{\left| x - y \right|^{s}} \right) \frac{dx\,dy}{\left| x - y \right|^{N}} \right] \\[4pt]
&\leq C(\epsilon)\left[ \max\left\{ \left\| \Psi(\underline{u}) \right\|_{W^{s,G}_{0}(\Omega)}^{p^{-}},\, \left\| \Psi(\underline{u}) \right\|_{W^{s,G}_{0}(\Omega)}^{p^{+}} \right\} + \max\left\{ \left\| w_0 \right\|_{W^{s,G}_{0}(\Omega)}^{p^{-}},\, \left\| w_0 \right\|_{W^{s,G}_{0}(\Omega)}^{p^{+}} \right\} \right] < \infty.
\end{align*}
On the one hand, consider a sequence \( (\varphi_{n}) \subset C^{\infty}_{c}(\Omega) \) such that \(\varphi_{n} \to \Psi_{m}\) strongly in \(W^{s,G}_{0}(\Omega)\), and define
\[
\tilde{\varphi}_{n} := \min \left\{ \Psi_{m}, \varphi_{n}^{+} \right\}.
\]
Clearly, \(\tilde{\varphi}_{n} \in W^{s,G}_{0}(\Omega) \cap L^{\infty}_{c}(\Omega)\). Consequently, we have
\[
\int_{\mathbb{R}^N} \int_{\mathbb{R}^N} 
g\left( \frac{|\underline{u}(x) - \underline{u}(y)|}{|x - y|^{s}} \right)
\frac{\underline{u}(x) - \underline{u}(y)}{|\underline{u}(x) - \underline{u}(y)|}
\cdot \frac{\tilde{\varphi}_n(x) - \tilde{\varphi}_n(y)}{|x - y|^{N+s}} \, dx \, dy
\leq \int_{\Omega} \left( f(x) \, \underline{u}^{-\alpha} + k(x) \, \underline{u}^{\beta} \right) \tilde{\varphi}_n(x) \, dx.
\]
From Definition~\ref{definition1}, and using the strong convergence \( \tilde{\varphi}_{n} \to \Psi_{m} \) in \( W^{s, G}_{0}(\Omega) \), we can pass to the limit as \( n \to \infty \), thereby obtaining
\begin{equation}\label{equ8}
\begin{aligned}
\int_{\mathbb{R}^N} \int_{\mathbb{R}^N} &
g\left( \frac{|\underline{u}(x) - \underline{u}(y)|}{|x - y|^{s}} \right)
\frac{\underline{u}(x) - \underline{u}(y)}{|\underline{u}(x) - \underline{u}(y)|}
\cdot \frac{\Psi_m(x) - \Psi_m(y)}{|x - y|^{N+s}} \, dx \, dy \\
& \leq \int_{\Omega} \left( f(x) \, \underline{u}^{-\alpha} + k(x) \, \underline{u}^{\beta} \right) \Psi_m(x) \, dx.
\end{aligned}
\end{equation}
On the other hand, testing \eqref{3comequ10} with \(\Phi_{m}\) leads to
\begin{equation}\label{equ7}
\begin{aligned}
\int_{\mathbb{R}^N} \int_{\mathbb{R}^N} &
g\left( \frac{|w_0(x) - w_0(y)|}{|x - y|^{s}} \right)
\frac{w_0(x) - w_0(y)}{|w_0(x) - w_0(y)|}
\cdot \frac{\Phi_m(x) - \Phi_m(y)}{|x - y|^{N+s}} \, dx \, dy \\
& \geq \int_{\Omega} \left[ f(x) (w_0 + \epsilon)^{-\alpha} + k(x) (w_0 + \epsilon)^{\beta} \right] \Phi_m(x) \, dx.
\end{aligned}
\end{equation}
Subtracting \eqref{equ8} from \eqref{equ7} yields
\begin{equation}\label{comequ1}
\begin{aligned}
\int_{\mathbb{R}^N} \int_{\mathbb{R}^N} \mathbf{E}_1(x, y) \, dx \, dy \leq \int_{\Omega} \mathbf{E}_2(x, y) \, dx,
\end{aligned}
\end{equation}
where
\begin{equation*}
\begin{aligned}
\mathbf{E}_1(x, y) &:= 
g\left( \frac{|\underline{u}(x) - \underline{u}(y)|}{|x - y|^{s}} \right)
\frac{\underline{u}(x) - \underline{u}(y)}{|\underline{u}(x) - \underline{u}(y)|}
\cdot \frac{\Psi_m(x) - \Psi_m(y)}{|x - y|^{N+s}} \\[4pt]
& \quad - 
g\left( \frac{|w_0(x) - w_0(y)|}{|x - y|^{s}} \right)
\frac{w_0(x) - w_0(y)}{|w_0(x) - w_0(y)|}
\cdot \frac{\Phi_m(x) - \Phi_m(y)}{|x - y|^{N+s}},
\\[8pt]
\mathbf{E}_2(x, y) &:=\left[ f(x) \left( \frac{\underline{u}^{-\alpha}}{(\underline{u} + m)^{p^{-}-1}} - \frac{(w_0 + \epsilon)^{-\alpha}}{(w_0 + m + \epsilon)^{p^{-}-1}} \right) \right. \\[4pt]
& \quad \left. + k(x) \left( \frac{\underline{u}^{\beta}}{(\underline{u} + m)^{p^{-}-1}} - \frac{(w_0 + \epsilon)^{\beta}}{(w_0 + m + \epsilon)^{p^{-}-1}} \right) \right] 
\mathbf{T}_\mathfrak{k} \left( \bigl( (\underline{u} + m)^{p^{-}} - (w_0 + m + \epsilon)^{p^{-}} \bigr)^+ \right).
\end{aligned}
\end{equation*}

\medskip

\noindent We begin with the left-hand side of \eqref{comequ1}, which can be decomposed as follows:
\begin{equation*}
\begin{aligned}
\int_{\mathbb{R}^N} \int_{\mathbb{R}^N} \mathbf{E}_1 (x, y) \, dx \, dy &= \iint_{\Omega_{1}\times\Omega_{1}} \mathbf{E}_1 (x, y)  \, dx \, dy + \iint_{\Omega_{2}\times\Omega_{2}} \mathbf{E}_1 (x, y)  \, dx \, dy + \iint_{\Omega_{3}\times\Omega_{3}} \mathbf{E}_1 (x, y)  \, dx \, dy \\
& + 2\iint_{\Omega_{2}\times \Omega_{1}} \mathbf{E}_1 (x, y)  \, dx \, dy + 2\iint_{\Omega_{3}\times \Omega_{1}} \mathbf{E}_1 (x, y)  \, dx \, dy + 2\iint_{\Omega_{3}\times \Omega_{2}} \mathbf{E}_1 (x, y)  \, dx \, dy.
\end{aligned}
\end{equation*}

\noindent We now estimate each term individually, beginning with the following observation:
\begin{equation*}
\Psi_{m} (x) - \Psi_{m} (y) = 0 \quad \text{and} \quad \Phi_{m} (x) - \Phi_{m} (y) = 0 \quad \text{for all } (x, y) \in \Omega_{1} \times \Omega_{1},
\end{equation*}
which implies that the integral over $\Omega_{1} \times \Omega_{1}$ vanishes, that is,
\begin{equation}\label{equ11}
\iint_{\Omega_{1}\times \Omega_{1}}  \mathbf{E}_1 (x, y) \, dx \, dy = 0.
\end{equation}

\noindent By applying Lemma~\ref{Lem2}, we derive the following inequality:
{\small
	\begin{equation}\label{equ12}
	\begin{aligned}
	& \iint_{\Omega_{2} \times \Omega_{2}} 
	g\left( \frac{\left| \underline{u}(x) - \underline{u}(y) \right|}{|x - y|^{s}} \right)  
	\frac{ \underline{u}(x) - \underline{u}(y) }{ \left| \underline{u}(x) - \underline{u}(y) \right| } \\
	& \quad \times \left( 
	\frac{ (\underline{u} + m)^{p^{-}}(x) - (w_{0} + m + \epsilon)^{p^{-}}(x) }{ (\underline{u} + m)^{p^{-} - 1}(x) }
	- \frac{ (\underline{u} + m)^{p^{-}}(y) - (w_{0} + m + \epsilon)^{p^{-}}(y) }{ (\underline{u} + m)^{p^{-} - 1}(y) }
	\right)
	\frac{dx\,dy}{|x - y|^{N + s}} \\[4pt]
	& \quad + \iint_{\Omega_{2} \times \Omega_{2}} 
	g\left( \frac{ \left| w_{0}(x) - w_{0}(y) \right| }{ |x - y|^{s} } \right)  
	\frac{ w_{0}(x) - w_{0}(y) }{ \left| w_{0}(x) - w_{0}(y) \right| } \\
	& \quad \times \left( 
	\frac{ (w_{0} + m + \epsilon)^{p^{-}}(x) - (\underline{u} + m)^{p^{-}}(x) }{ (w_{0} + m + \epsilon)^{p^{-} - 1}(x) }
	- \frac{ (w_{0} + m + \epsilon)^{p^{-}}(y) - (\underline{u} + m)^{p^{-}}(y) }{ (w_{0} + m + \epsilon)^{p^{-} - 1}(y) }
	\right)
	\frac{dx\,dy}{|x - y|^{N + s}} \geq 0.
	\end{aligned}
	\end{equation}
}
	
\noindent 
Moreover, by invoking Lagrange's theorem and using \eqref{equ2}, we derive the following estimate:
{\small
	\begin{equation*}
	\begin{aligned}
	&\iint_{\Omega_{3} \times \Omega_{3}} E_{1}(x, y) \, dx \, dy \\[4pt]
	&  \geq -\mathfrak{k}(p^{-}-1) \iint_{\Omega_{3} \times \Omega_{3}} g\left( \dfrac{\left| \underline{u}(x) - \underline{u}(y) \right| }{\left| x - y \right|^{s}} \right)  
	\dfrac{\left| \underline{u}(x) - \underline{u}(y) \right| }{\left| x - y \right|^{s}}  \left( \dfrac{ \max\left\lbrace (\underline{u}(y) + m)^{p^{-}-2},\, (\underline{u}(x) + m)^{p^{-}-2} \right\rbrace }
	{(\underline{u}(y) + m)^{p^{-}-1} (\underline{u}(x) + m)^{p^{-}-1}} \right) \dfrac{dx \, dy}{\left| x - y \right|^{N}} \\[4pt]
	& \geq -\mathfrak{k} p^{+} (p^{-}-1) \iint_{\Omega_{3} \times \Omega_{3}} G\left( \dfrac{\left| \underline{u}(x) - \underline{u}(y) \right| }{\left| x - y \right|^{s}} \right)  
	\left( \dfrac{ \max\left\lbrace (\underline{u}(y) + m)^{p^{-}-2},\, (\underline{u}(x) + m)^{p^{-}-2} \right\rbrace }
	{(\underline{u}(y) + m)^{p^{-}-1} (\underline{u}(x) + m)^{p^{-}-1}} \right) \dfrac{dx \, dy}{\left| x - y \right|^{N}}.
	\end{aligned}
	\end{equation*}
}

\noindent
Now, define
\[
\omega_{\underline{u}} := \left\lbrace (x, y) \in \Omega_{3} \times \Omega_{3} \,\middle|\, \underline{u}(x) > \underline{u}(y) \right\rbrace.
\]
Taking into account that $\operatorname{meas}(\Omega_{3}) \to 0$ as $\mathfrak{k} \to \infty$, we obtain
\begin{equation} \label{equ13}
\begin{aligned}
\iint_{\Omega_{3} \times \Omega_{3}} E_{1}(x, y) \, dx \, dy & \geq - p^{+}(p^{-} - 1) \iint_{(\Omega_{3} \times \Omega_{3}) \cap \omega_{\underline{u}}} G\left( \dfrac{\left| \underline{u}(x) - \underline{u}(y) \right| }{\left| x - y \right|^{s}} \right) 
\left( \dfrac{\underline{u}(y) + m}{\underline{u}(x) + m} \right) \dfrac{dx \, dy}{\left| x - y \right|^{N}} \\[4pt]
& \quad - p^{+}(p^{-} - 1) \iint_{(\Omega_{3} \times \Omega_{3}) \cap \omega_{\underline{u}}^{c}} G\left( \dfrac{\left| \underline{u}(x) - \underline{u}(y) \right| }{\left| x - y \right|^{s}} \right) 
\left( \dfrac{\underline{u}(x) + m}{\underline{u}(y) + m} \right) \dfrac{dx \, dy}{\left| x - y \right|^{N}} \\[4pt]
& \geq -p^{+} (p^{-} - 1) \iint_{\Omega_{3} \times \Omega_{3}} G\left( \dfrac{\left| \underline{u}(x) - \underline{u}(y) \right| }{\left| x - y \right|^{s}} \right) \dfrac{dx \, dy}{\left| x - y \right|^{N}}
\to 0 \quad \text{as } \mathfrak{k} \to \infty,
\end{aligned}
\end{equation}
where we have used \cite[Lemma 3.11]{bal2024singular} and Definition~\ref{definition1} to ensure the finiteness of the last integral. Moreover, it is easy to observe, by assumption \textbf{(H4)}, that the following function is nondecreasing:
\[
\tau \mapsto g\left( \left| \tau - \tau_{0} \right| \right) \frac{\tau - \tau_{0}}{\left| \tau - \tau_{0} \right|} \cdot \frac{1}{ \tau^{p^{-} - 1}},
\]
which leads to the following inequality:

{\small
	\begin{equation}\label{equ10}
	\begin{aligned}
	\iint_{\Omega_{2} \times \Omega_{1}} E_{1}(x, y)\, dx \, dy 
	&= \iint_{\Omega_{2} \times \Omega_{1}} \frac{(\underline{u}(x) + m)^{p^{-}} - (w_{0}(x) + m + \epsilon)^{p^{-}}}{|x - y|^{N + s }} \\
	&\times \Bigg[ 
	g\left( \frac{\left| \underline{u}(x) - \underline{u}(y) \right|}{|x - y|^{s}} \right)  
	\frac{ \underline{u}(x) - \underline{u}(y) }{ \left| \underline{u}(x) - \underline{u}(y) \right| } \cdot \frac{1}{(\underline{u}(x) + m)^{p^{-}-1}} \\
	& - 
	g\left( \frac{\left| w_{0}(x) - w_{0}(y) \right|}{|x - y|^{s}} \right)  
	\frac{ w_{0}(x) - w_{0}(y) }{ \left| w_{0}(x) - w_{0}(y) \right| } \cdot \frac{1}{(w_{0}(x) + m + \epsilon)^{p^{-}-1}} 
	\Bigg] dx \, dy \geq 0.
	\end{aligned}
	\end{equation}
}

\noindent Similarly, we obtain
{\small \begin{equation}\label{equ14}
\begin{aligned}
&\iint_{\Omega_{3} \times \Omega_{1}} E_{1}(x, y)\, dx \, dy 
= \iint_{\Omega_{3} \times \Omega_{1}} \frac{k}{|x - y|^{N + s }}\Bigg[ 
g\left( \frac{\left| \underline{u}(x) - \underline{u}(y) \right|}{|x - y|^{s}} \right)  
\frac{ \underline{u}(x) - \underline{u}(y) }{ \left| \underline{u}(x) - \underline{u}(y) \right| } \cdot \frac{1}{(\underline{u}(x) + m)^{p^{-}-1}} \\
&\qquad \quad - 
g\left( \frac{\left| w_{0}(x) - w_{0}(y) \right|}{|x - y|^{s}} \right)  
\frac{ w_{0}(x) - w_{0}(y) }{ \left| w_{0}(x) - w_{0}(y) \right| } \cdot \frac{1}{(w_{0}(x) + m + \epsilon)^{p^{-}-1}} 
\Bigg] dx \, dy \geq 0.
\end{aligned}
\end{equation}}

\noindent
To estimate the last integral, we observe that
\begin{equation*}
\iint_{\Omega_{3} \times \Omega_{2}} E_{1}(x, y) \, dx \, dy 
= \mathbf{I}_{1} + \mathbf{I}_{2},
\end{equation*}
where
\begin{equation*}
\begin{aligned}
\mathbf{I}_{1} := \iint_{\Omega_{3} \times \Omega_{2}} & 
g\left( \frac{ \left| \underline{u}(x) - \underline{u}(y) \right| }{ |x - y|^{s} } \right)  
\frac{ \underline{u}(x) - \underline{u}(y) }{ \left| \underline{u}(x) - \underline{u}(y) \right| } \\
& \times \left( \frac{k}{(\underline{u}(x) + m)^{p^{-}-1}} 
- \frac{ (\underline{u}(y) + m)^{p^{-}} - (w_{0}(y) + m + \epsilon)^{p^{-}} }{ (\underline{u}(y) + m)^{p^{-}-1} } \right) 
\frac{dx \, dy}{|x - y|^{N + s}},
\end{aligned}
\end{equation*}
\begin{equation*}
\begin{aligned}
\mathbf{I}_{2} := - \iint_{\Omega_{3} \times \Omega_{2}} & 
g\left( \frac{ \left| w_{0}(x) - w_{0}(y) \right| }{ |x - y|^{s} } \right)  
\frac{ w_{0}(x) - w_{0}(y) }{ \left| w_{0}(x) - w_{0}(y) \right| } \\
& \times \left( \frac{k}{(w_{0}(x) + m + \epsilon)^{p^{-}-1}} 
- \frac{ (\underline{u}(y) + m)^{p^{-}} - (w_{0}(y) + m + \epsilon)^{p^{-}} }{ (w_{0}(y) + m + \epsilon)^{p^{-}-1} } \right) 
\frac{dx \, dy}{|x - y|^{N + s}}.
\end{aligned}
\end{equation*}

\noindent 
We begin the analysis of $\mathbf{I}_{1}$ by introducing the set
\[
\mathbf{S}_{\underline{u}} := \left\lbrace (x, y) \in \Omega_{3} \times \Omega_{2} \, : \, \underline{u}(x) > \underline{u}(y) \right\rbrace.
\]

\noindent
By invoking identity~\eqref{equ2} and applying the Lagrange mean value theorem, we derive the following lower bound:
\begin{equation*}
\begin{aligned}
\mathbf{I}_1 \geq\; & \iint_{\mathbf{S}_{\underline{u}}} 
g\left( \frac{ \left| \underline{u}(x) - \underline{u}(y) \right| }{ |x - y|^{s} } \right)  
\frac{ \underline{u}(x) - \underline{u}(y) }{ \left| \underline{u}(x) - \underline{u}(y) \right| }  \\[6pt]
& \qquad \times \left[ \frac{ (\underline{u}(y) + m)^{p^{-}} - (w_{0}(y) + m + \epsilon)^{p^{-}} }{ (\underline{u}(x) + m)^{p^{-} - 1} } 
- \frac{ (\underline{u}(y) + m)^{p^{-}} - (w_{0}(y) + m + \epsilon)^{p^{-}} }{ (\underline{u}(y) + m)^{p^{-} - 1} } \right]
\frac{dx \, dy}{|x - y|^{N + s}}  \\[6pt]
& + \iint_{\mathbf{S}_{\underline{u}}^{\mathrm{c}}} 
g\left( \frac{ \left| \underline{u}(x) - \underline{u}(y) \right| }{ |x - y|^{s} } \right)  
\frac{ \underline{u}(x) - \underline{u}(y) }{ \left| \underline{u}(x) - \underline{u}(y) \right| } \\[6pt]
& \qquad \times \left[ \frac{ (\underline{u}(x) + m)^{p^{-}} - (w_{0}(x) + m + \epsilon)^{p^{-}} }{ (\underline{u}(x) + m)^{p^{-} - 1} } 
- \frac{ (\underline{u}(y) + m)^{p^{-}} - (w_{0}(y) + m + \epsilon)^{p^{-}} }{ (\underline{u}(y) + m)^{p^{-} - 1} } \right]
\frac{dx \, dy}{|x - y|^{N + s}} \\[6pt]
\geq\; & - p^{+}(p^{-} - 1) \iint_{\mathbf{S}_{\underline{u}}} 
G\left( \frac{ \left| \underline{u}(x) - \underline{u}(y) \right| }{ \left| x - y \right|^{s} } \right) 
\frac{dx \, dy}{\left| x - y \right|^{N}} \\[6pt]
& \underbrace{
	- \iint_{\mathbf{S}_{\underline{u}}^{\mathrm{c}}} 
	g\left( \frac{ \left| \underline{u}(x) - \underline{u}(y) \right| }{ |x - y|^{s} } \right)  
	\frac{ \underline{u}(x) - \underline{u}(y) }{ \left| \underline{u}(x) - \underline{u}(y) \right| } 
	\left[ \frac{ (w_{0}(x) + m + \epsilon)^{p^{-}} }{ (\underline{u}(x) + m)^{p^{-} - 1} } 
	- \frac{ (w_{0}(y) + m + \epsilon)^{p^{-}} }{ (\underline{u}(y) + m)^{p^{-} - 1} } \right] 
	\frac{dx \, dy}{\left| x - y \right|^{N + s}}
}_{\textbf{(*)}}.
\end{aligned}
\end{equation*}

\noindent
To estimate term \textbf{(*)}, we introduce the following subsets of \( \mathbf{S}_{\underline{u}}^{\mathrm{c}} \):
\[
\mathbf{S}^{1}_{\underline{u}}{}^{\mathrm{c}} := \left\lbrace (x, y) \in \mathbf{S}_{\underline{u}}^{\mathrm{c}} \;\middle|\;
\frac{ \left| \underline{u}(x) - \underline{u}(y) \right|}{ |x - y|^{s} } < 1 \;\text{and}\;
\frac{ \left| w_{0}(x) - w_{0}(y) \right|}{ |x - y|^{s} } < 1 
\right\rbrace,
\]
\[
\mathbf{S}^{2}_{\underline{u}}{}^{\mathrm{c}} := \left\lbrace (x, y) \in \mathbf{S}_{\underline{u}}^{\mathrm{c}} \;\middle|\;
\frac{ \left| \underline{u}(x) - \underline{u}(y) \right|}{ |x - y|^{s} } < 1 \;\text{and}\;
\frac{ \left| w_{0}(x) - w_{0}(y) \right|}{ |x - y|^{s} } > 1 
\right\rbrace,
\]
\[
\mathbf{S}^{3}_{\underline{u}}{}^{\mathrm{c}} := \left\lbrace (x, y) \in \mathbf{S}_{\underline{u}}^{\mathrm{c}} \;\middle|\;
\frac{ \left| \underline{u}(x) - \underline{u}(y) \right|}{ |x - y|^{s} } > 1 \;\text{and}\;
\frac{ \left| w_{0}(x) - w_{0}(y) \right|}{ |x - y|^{s} } < 1 
\right\rbrace,
\]
\[
\mathbf{S}^{4}_{\underline{u}}{}^{\mathrm{c}} := \left\lbrace (x, y) \in \mathbf{S}_{\underline{u}}^{\mathrm{c}} \;\middle|\;
\frac{ \left| \underline{u}(x) - \underline{u}(y) \right|}{ |x - y|^{s} } > 1 \;\text{and}\;
\frac{ \left| w_{0}(x) - w_{0}(y) \right|}{ |x - y|^{s} } > 1 
\right\rbrace.
\]

\noindent
By combining relation~\eqref{equ2} with the \( G \)-fractional Picone inequality (see Proposition~\ref{proposition2}), and subsequently applying Young's inequality, we derive the following lower bound:
\begin{equation*}
\begin{aligned}
\textbf{(*)} \geq\; & - C \Bigg[
\iint_{\mathbf{S}_{\underline{u}}^{1\,\mathrm{c}}} 
	G\left( \frac{|w_{0}(x) - w_{0}(y)|}{|x - y|^{s}} \right)^{\frac{p^{-}}{p^{+}}} 
	\frac{dx \, dy}{|x - y|^{N}}\\
& \quad + \iint_{\mathbf{S}_{\underline{u}}^{2\,\mathrm{c}}} 
G\left( \frac{|w_{0}(x) - w_{0}(y)|}{|x - y|^{s}} \right) 
\frac{dx \, dy}{|x - y|^{N}} \\
& \quad + \iint_{\mathbf{S}_{\underline{u}}^{1\,\mathrm{c}}} 
	G\left( \frac{|w_{0}(x) - w_{0}(y)|}{|x - y|^{s}} \right)^{\frac{p^{-}}{p^{+}}} 	
	G\left( \frac{|\underline{u}(x) - \underline{u}(y)|}{|x - y|^{s}} \right)^{\frac{p^{+} - p^{-}}{p^{-}}} 
	\frac{dx \, dy}{|x - y|^{N}}\\
& \quad + \iint_{\mathbf{S}_{\underline{u}}^{1\,\mathrm{c}}} 
	G\left( \frac{|w_{0}(x) - w_{0}(y)|}{|x - y|^{s}} \right)	
	G\left( \frac{|\underline{u}(x) - \underline{u}(y)|}{|x - y|^{s}} \right)^{\frac{p^{+} - p^{-}}{p^{-}}} 
	\frac{dx \, dy}{|x - y|^{N}}
\Bigg],
\end{aligned}
\end{equation*}
\noindent
where the constant \( C \) depends only on \( p^{-} \), \( p^{+} \), and \( G(1) \). Referring to \( \mathbf{I}_1 \), we obtain

\begin{equation*}
\begin{aligned}
\mathbf{I}_1 \geq\; & - p^{+}(p^{-} - 1) \iint_{\mathbf{S}_{\underline{u}}} 
G\left( \frac{ \left| \underline{u}(x) - \underline{u}(y) \right| }{ \left| x - y \right|^{s} } \right) 
\frac{dx \, dy}{\left| x - y \right|^{N}} \\[6pt]
& - C \Bigg[
\iint_{\mathbf{S}_{\underline{u}}^{1\,\mathrm{c}}} 
	G\left( \frac{|w_{0}(x) - w_{0}(y)|}{|x - y|^{s}} \right)^{\frac{p^{-}}{p^{+}}} 
	\frac{dx \, dy}{|x - y|^{N}}\\[6pt]
& \quad + \iint_{\mathbf{S}_{\underline{u}}^{2\,\mathrm{c}}} 
G\left( \frac{|w_{0}(x) - w_{0}(y)|}{|x - y|^{s}} \right) 
\frac{dx \, dy}{|x - y|^{N}} \\[6pt]
& \quad + \iint_{\mathbf{S}_{\underline{u}}^{1\,\mathrm{c}}} 
	G\left( \frac{|w_{0}(x) - w_{0}(y)|}{|x - y|^{s}} \right)^{\frac{p^{-}}{p^{+}}} 	
	G\left( \frac{|\underline{u}(x) - \underline{u}(y)|}{|x - y|^{s}} \right)^{\frac{p^{+} - p^{-}}{p^{-}}} 
	\frac{dx \, dy}{|x - y|^{N}}\\[6pt]
& \quad + \iint_{\mathbf{S}_{\underline{u}}^{1\,\mathrm{c}}} 
	G\left( \frac{|w_{0}(x) - w_{0}(y)|}{|x - y|^{s}} \right)	
	G\left( \frac{|\underline{u}(x) - \underline{u}(y)|}{|x - y|^{s}} \right)^{\frac{p^{+} - p^{-}}{p^{-}}} 
	\frac{dx \, dy}{|x - y|^{N}}
\Bigg].
\end{aligned}
\end{equation*}

\noindent
We also define the set
\[
\mathbf{S}_{w_{0}} := \left\lbrace (x, y) \in \Omega_{3} \times \Omega_{2} \, : \, w_{0}(x) > w_{0}(y) \right\rbrace,
\]
which allows us to derive the following estimate for $\mathbf{I}_{2}$:
\begin{equation*}
\begin{aligned}
\mathbf{I}_{2} \geq & - \iint_{\mathbf{S}_{w_{0}}} 
g\left( \frac{ |w_{0}(x) - w_{0}(y)| }{ |x - y|^{s} } \right)  
\frac{ w_{0}(x) - w_{0}(y) }{ |w_{0}(x) - w_{0}(y)| } \\
& \quad \times \left( \frac{ (\underline{u}(x) + m)^{p^{-}} - (w_{0}(x) + m + \epsilon)^{p^{-}} }{(w_{0}(x) + m + \epsilon)^{p^{-}-1}} 
- \frac{ (\underline{u}(y) + m)^{p^{-}} - (w_{0}(y) + m + \epsilon)^{p^{-}} }{ (w_{0}(y) + m + \epsilon)^{p^{-}-1} } \right) 
\frac{dx \, dy}{|x - y|^{N + s}} \\[6pt]
& - \iint_{\mathbf{S}_{w_{0}}^{c}} 
g\left( \frac{ |w_{0}(x) - w_{0}(y)| }{ |x - y|^{s} } \right)  
\frac{ w_{0}(x) - w_{0}(y) }{ |w_{0}(x) - w_{0}(y)| } \\
& \quad \times \left( \frac{k}{(w_{0}(x) + m + \epsilon)^{p^{-}-1}} 
- \frac{k}{ (w_{0}(y) + m + \epsilon)^{p^{-}-1} } \right) 
\frac{dx \, dy}{|x - y|^{N + s}}\\
& \geq - C \Bigg[
\iint_{\mathbf{S}_{w_{0}}^{1\,\mathrm{c}}} 
	G\left( \frac{|\underline{u}(x) - \underline{u}(y)|}{|x - y|^{s}} \right)^{\frac{p^{-}}{p^{+}}} 
	\frac{dx \, dy}{|x - y|^{N}}\\
& \quad + \iint_{\mathbf{S}_{w_{0}}^{2\,\mathrm{c}}} 
G\left( \frac{|\underline{u}(x) - \underline{u}(y)|}{|x - y|^{s}} \right) 
\frac{dx \, dy}{|x - y|^{N}} \\
& \quad + \iint_{\mathbf{S}_{w_{0}}^{1\,\mathrm{c}}} 
	G\left( \frac{|\underline{u}(x) - \underline{u}(y)|}{|x - y|^{s}} \right)^{\frac{p^{-}}{p^{+}}} 	
	G\left( \frac{|w_{0}(x) - w_{0}(y)|}{|x - y|^{s}} \right)^{\frac{p^{+} - p^{-}}{p^{-}}} 
	\frac{dx \, dy}{|x - y|^{N}}\\
& \quad +  \iint_{\mathbf{S}_{w_{0}}^{1\,\mathrm{c}}} 
	G\left( \frac{|\underline{u}(x) - \underline{u}(y)|}{|x - y|^{s}} \right)	
	G\left( \frac{|w_{0}(x) - w_{0}(y)|}{|x - y|^{s}} \right)^{\frac{p^{+} - p^{-}}{p^{-}}} 
	\frac{dx \, dy}{|x - y|^{N}}
\Bigg],
\end{aligned}
\end{equation*}
where 
\[
\mathbf{S}^{1}_{w_{0}}{}^{\mathrm{c}} := \left\lbrace (x, y) \in \mathbf{S}_{\underline{u}}^{\mathrm{c}} \;\middle|\;
\frac{ \left| \underline{u}(x) - \underline{u}(y) \right|}{ |x - y|^{s} } < 1 \;\text{and}\;
\frac{ \left| w_{0}(x) - w_{0}(y) \right|}{ |x - y|^{s} } < 1 
\right\rbrace,
\]

\[
\mathbf{S}^{2}_{w_{0}}{}^{\mathrm{c}} := \left\lbrace (x, y) \in \mathbf{S}_{\underline{u}}^{\mathrm{c}} \;\middle|\;
\frac{ \left| \underline{u}(x) - \underline{u}(y) \right|}{ |x - y|^{s} } < 1 \;\text{and}\;
\frac{ \left| w_{0}(x) - w_{0}(y) \right|}{ |x - y|^{s} } > 1 
\right\rbrace,
\]

\[
\mathbf{S}^{3}_{w_{0}}{}^{\mathrm{c}} := \left\lbrace (x, y) \in \mathbf{S}_{\underline{u}}^{\mathrm{c}} \;\middle|\;
\frac{ \left| \underline{u}(x) - \underline{u}(y) \right|}{ |x - y|^{s} } > 1 \;\text{and}\;
\frac{ \left| w_{0}(x) - w_{0}(y) \right|}{ |x - y|^{s} } < 1 
\right\rbrace,
\]

\[
\mathbf{S}^{4}_{w_{0}}{}^{\mathrm{c}} := \left\lbrace (x, y) \in \mathbf{S}_{\underline{u}}^{\mathrm{c}} \;\middle|\;
\frac{ \left| \underline{u}(x) - \underline{u}(y) \right|}{ |x - y|^{s} } > 1 \;\text{and}\;
\frac{ \left| w_{0}(x) - w_{0}(y) \right|}{ |x - y|^{s} } > 1 
\right\rbrace.
\]
\noindent
Then, we have
\begin{equation}\label{equ15}
\iint_{\Omega_{3} \times \Omega_{2}} E_{1}(x, y) \, dx \, dy \geq 0 \quad \text{as} \quad k \to \infty.
\end{equation}

\noindent
Hence, by combining \eqref{equ11}--\eqref{equ15}, we deduce
\begin{equation}\label{equ16}
\int_{\mathbb{R}^N} \int_{\mathbb{R}^N} \mathbf{E}_1(x, y) \, dx \, dy \geq 0 \quad \text{as} \quad k \to \infty.
\end{equation}

\noindent
Next, consider the right-hand side of \eqref{comequ1}. By exploiting the monotonicity property, we observe that
	\begin{equation}\label{equ17}
	\begin{aligned}
	\int_{\Omega} \mathbf{E}_2(x, y) \, dx &\leq \int_{\mathcal{S}_{\epsilon}} k(x) \left( \frac{\underline{u}^{\beta}}{(\underline{u} + m)^{p^{-}-1}} - \frac{(w_{0} + \epsilon)^{\beta}}{(w_{0} + m + \epsilon)^{p^{-}-1}} \right) \\
	&\quad \times \textbf{T}_{\mathfrak{k}} \left( \left[ (\underline{u} + m)^{p^{-}} - (w_{0} + m + \epsilon)^{p^{-}} \right]^+ \right) dx.
	\end{aligned}
	\end{equation}

\noindent
To pass to the limit on the right-hand side of inequality \eqref{equ17}, we adapt some ideas from \cite[Theorem 2.2]{durastanti2022comparison}. First, observe that, on the one hand, by using the inequality \( \textbf{T}_{\mathfrak{k}}(s) \leq s \) for all \( s \geq 0 \) and applying the Mean Value Theorem, we obtain
\begin{equation*}
\begin{aligned}
&k(x)\left( \frac{\underline{u}(x)^{\beta}}{(\underline{u}(x) + m)^{p^{-}-1}} - \frac{(w_{0}(x) + \epsilon)^{\beta}}{(w_{0}(x) + m + \epsilon)^{p^{-}-1}} \right)^+ \times \textbf{T}_{\mathfrak{k}} \left( \left[ (\underline{u}(x) + m)^{p^{-}} - (w_{0}(x) + m + \epsilon)^{p^{-}} \right]^+ \right) \\
&\quad \leq k(x)\, \underline{u}(x)^{\beta} \left( \frac{(\underline{u}(x) + m)^{p^{-}} - (w_{0}(x) + m + \epsilon)^{p^{-}}}{(\underline{u}(x) + m)^{p^{-}-1}} \right) \\
&\quad \leq k(x)\, \underline{u}(x)^{\beta} \left( \frac{(\underline{u}(x) + m)^{p^{-}} - m^{p^{-}}}{(\underline{u}(x) + m)^{p^{-}-1}} \right) \\
&\quad \leq p^{-} k(x)\, \underline{u}(x)^{\beta + 1} \in L^1(\mathcal{S}_{\epsilon}),
\end{aligned}
\end{equation*}

\noindent
By applying the Dominated Convergence Theorem, we obtain
\begin{equation}\label{equ18}
\begin{aligned}
\lim_{m \to 0} &\int_{\mathcal{S}_{\epsilon}} k(x) \left( \frac{\underline{u}(x)^{\beta}}{(\underline{u}(x) + m)^{p^{-}-1}} - \frac{(w_{0}(x) + \epsilon)^{\beta}}{(w_{0}(x) + m + \epsilon)^{p^{-}-1}} \right)^{+} \\
&\qquad \times \textbf{T}_{\mathfrak{k} } \left( \left[ (\underline{u}(x) + m)^{p^{-}} - (w_{0}(x) + m + \epsilon)^{p^{-}} \right]^+ \right) \, dx \\
&= \int_{\mathcal{S}_{\epsilon}} k(x) \left( \underline{u}(x)^{\beta - p^{-} + 1} - (w_{0}(x) + \epsilon)^{\beta - p^{-} + 1} \right)^{+} \\
&\qquad \times \textbf{T}_{\mathfrak{k}} \left( \left[ \underline{u}(x)^{p^{-}} - (w_{0}(x) + \epsilon)^{p^{-}} \right]^+ \right) \, dx.
\end{aligned}
\end{equation}
On the other hand, by applying Fatou’s Lemma, we conclude
\begin{equation}\label{equ19}
\begin{aligned}
\liminf_{m \to 0} &\int_{\mathcal{S}_{\epsilon}} k(x) \left( \frac{\underline{u}(x)^{\beta}}{(\underline{u}(x) + m)^{p^{^-}-1}} - \frac{(w_{0}(x) + \epsilon)^{\beta}}{(w_{0}(x) + m + \epsilon)^{p^{-}-1}} \right)^{-} \\
&\qquad \times \textbf{T}_{\mathfrak{k} } \left( \left[ (\underline{u}(x) + m)^{p^{-}} - (w_{0}(x) + m + \epsilon)^{p^{-}} \right]^+ \right) \, dx \\
&\geq \int_{\mathcal{S}_{\epsilon}} k(x) \left( \underline{u}(x)^{\beta - p^{-} + 1} - (w_{0}(x) + \epsilon)^{\beta - p^{-} + 1} \right)^{-} \\
&\qquad \times \textbf{T}_{\mathfrak{k} } \left( \left[ \underline{u}(x)^{p^{-}} - (w_{0}(x) + \epsilon)^{p^{-}} \right]^+ \right) \, dx.
\end{aligned}
\end{equation}

\noindent 
Taking the $\liminf$ in \eqref{equ17} as $m \to 0$, and using \eqref{equ18} and \eqref{equ19}, we obtain
\begin{equation*}
\begin{aligned}
\liminf_{m \to 0} \int_{\Omega} \mathbf{E}_2(x, y) \, dx 
&\leq \int_{\Omega} k(x) \left(\underline{u}^{\beta - p^{-} + 1} - (w_0 + \epsilon)^{\beta - p^{-} + 1} \right) 
\mathbf{T}_{\mathfrak{k} } \left( \left(\underline{u}^{p^{-}} - (w_0 + \epsilon)^{p^{-}} \right)^{+} \right) dx.
\end{aligned}
\end{equation*}

\noindent
 Moreover, due to the monotonicity, we have
\[
U_{\mathfrak{k} } := - k(x) \left(\underline{u}^{\beta - p^{-} + 1} - (w_0 + \epsilon)^{\beta -  p^{-} + 1} \right) 
\mathbf{T}_{\mathfrak{k} } \left( \left( \underline{u}^{p^{-}} - (w_0 + \epsilon)^{p^{-}} \right)^{+} \right) \geq 0.
\]

\noindent 
Since the sequence $(U_\mathfrak{k} )$ is increasing, the Monotone Convergence Theorem yields
\[
\lim_{\mathfrak{k}  \to \infty} \int_{\Omega} U_\mathfrak{k}  \, dx 
= - \int_{\Omega} k(x) \left( \underline{u}^{\beta - p^{-} + 1} - (w_0 + \epsilon)^{\beta - p^{-} + 1} \right) 
\left( \underline{u}^{p^{-}} - (w_0 + \epsilon)^{p^{-}} \right)^{+} dx.
\]

\noindent
Therefore, since $\beta < p^{-} - 1$, we conclude that
{\small \begin{equation}\label{equ20}
\begin{aligned}
\lim_{\mathfrak{k}  \to \infty} \liminf_{m \to 0} \int_{\Omega} \mathbf{E}_2(x, y) \, dx 
&\leq \int_{\Omega} k(x) \left( \underline{u}^{\beta - p^{-} + 1} - (w_0 + \epsilon)^{\beta - p^{-} + 1} \right) 
\left( \underline{u}^{p^{-}} - (w_0 + \epsilon)^{p^{-}} \right)^{+} dx \leq 0.
\end{aligned}
\end{equation}}

\noindent
 Finally, combining \eqref{equ16}, and \eqref{equ20}, we deduce
\begin{equation*}
\begin{aligned}
0 \leq \int_{\Omega} k(x) \left( \underline{u}^{\beta - p^{-} + 1} - (w_0 + \epsilon)^{\beta - p^{-} + 1} \right) 
\left( \underline{u}^{p^{-}} - (w_0 + \epsilon)^{p^{-}} \right)^{+} dx \leq 0.
\end{aligned}
\end{equation*}

\noindent
 Consequently, we obtain $\underline{u} \leq w_0 + \epsilon \leq \overline{u} + \epsilon$. Letting $\epsilon \to 0$, it follows that $\underline{u} \leq \overline{u}$.
\end{proof}

\noindent
We now present the proof of our second main result concerning problem~\eqref{GGP}. It is worth emphasizing that the function~$h$ may be unbounded both near the origin, due to a mild singularity, and at infinity. Furthermore, the sub and super-solution of~\eqref{GGP}, as defined in Definition~\ref{definition2}, may themselves exhibit unbounded behavior.

\begin{proof}[\textbf{Proof of Theorem~\ref{Theorem2}}]
	Let $\Upsilon, \Phi \in W^{s, G}_{0}(\Omega)$ be any pair of nonnegative functions. Then, the following inequalities hold:
	\begin{equation*}
	\int_{\mathbb{R}^{N}} \int_{\mathbb{R}^{N}} g\left( \frac{\left| \underline{u}(x) - \underline{u}(y) \right|}{\left| x - y \right|^{s}} \right) 
	\dfrac{\underline{u}(x) - \underline{u}(y)}{\left| \underline{u}(x) - \underline{u}(y) \right|} 
	\dfrac{\Upsilon(x) - \Upsilon(y)}{\left| x - y \right|^{N + s}} \, dx \, dy 
	\leq \int_{\Omega} F(x, \underline{u}) \Upsilon(x) \, dx,
	\end{equation*}
	and
	\begin{equation*}
	\int_{\mathbb{R}^{N}} \int_{\mathbb{R}^{N}} g\left( \frac{\left| \overline{u}(x) - \overline{u}(y) \right|}{\left| x - y \right|^{s}} \right) 
	\dfrac{\overline{u}(x) - \overline{u}(y)}{\left| \overline{u}(x) - \overline{u}(y) \right|} 
	\dfrac{\Phi(x) - \Phi(y)}{\left| x - y \right|^{N + s}} \, dx \, dy 
	\geq \int_{\Omega} F(x, \overline{u}) \Phi(x) \, dx.
	\end{equation*}
	Subtracting the above inequalities using the test functions
	\[
	\Upsilon = \frac{\textbf{T}_{\mathfrak{k} }\left( \left(\left(\underline{u} + m\right)^{p^{-}} - \left( \overline{u}+ m\right)^{p^{-}}\right)^{+} \right)}{(\underline{u} + m)^{p^{-} - 1}}, \quad 
	\Phi = \frac{\textbf{T}_{\mathfrak{k} }\left( \left( \left( \overline{u}+ m \right)^{p^{-}} - \left(\underline{u} + m\right)^{p^{-}}\right)^{-} \right)}{\left( \overline{u}+ m \right)^{p^{-} - 1}}, \quad m > 0,
	\]
	we obtain
\begin{equation}\label{equ22}
\begin{aligned}
&\int_{\mathbb{R}^{N}} \int_{\mathbb{R}^{N}} g\left( \frac{|\underline{u}(x) - \underline{u}(y)|}{|x - y|^{s}} \right)
\frac{\underline{u}(x) - \underline{u}(y)}{|\underline{u}(x) - \underline{u}(y)|}
\cdot \frac{\Upsilon(x) - \Upsilon(y)}{|x - y|^{N+s}} \, dx \, dy \\[4pt]
&\quad - \int_{\mathbb{R}^{N}} \int_{\mathbb{R}^{N}} g\left( \frac{|w_0(x) - w_0(y)|}{|x - y|^{s}} \right)
\frac{w_0(x) - w_0(y)}{|w_0(x) - w_0(y)|}
\cdot \frac{\Phi(x) - \Phi(y)}{|x - y|^{N+s}} \, dx \, dy \\[4pt]
&\leq \int_{\left\lbrace\underline{u} > \overline{u}\right\rbrace}  \left( \frac{h(x, \underline{u})}{(\underline{u} + m)^{p^{-}-1}} - \frac{h(x, \overline{u})}{(\overline{u} + m)^{p^{-}-1}}\right) 
\textbf{T}_{\mathfrak{k}} \left(\left(\underline{u} + m\right)^{p^{-}} - \left(\overline{u} + m\right)^{p^{-}} \right) \, dx.
\end{aligned}
\end{equation}
By arguments similar to those used in the proof of Claim~3 in Theorem~\ref{Theorem1}, we find that the left-hand side of \eqref{equ22} is nonnegative. Applying the Mean Value Theorem, we get
	\begin{equation*}
	\begin{aligned}
	&\left( \frac{h(x, \underline{u})}{(\underline{u} + m)^{p^{-}-1}} - \frac{h(x, \overline{u})}{(\overline{u} + m)^{p^{-}-1}} \right)^{+} 
	\textbf{T}_{\mathfrak{k}} \left( (\underline{u} + m)^{p^{-}} - (\overline{u} + m)^{p^{-}} \right) \chi_{\left\lbrace \underline{u} > \overline{u} \right\rbrace} \\
	&\quad \leq h(x, \underline{u}) \left( \frac{(\underline{u} + m)^{p^{-}} - (\overline{u} + m)^{p^{-}}}{(\underline{u} + m)^{p^{-}-1}} \right) \chi_{\left\lbrace \underline{u} > \overline{u} \right\rbrace} \\
	&\quad \leq h(x, \underline{u}) \left( \frac{(\underline{u} + m)^{p^{-}} - m^{p^{-}}}{(\underline{u} + m)^{p^{-}-1}} \right) \chi_{\left\lbrace \underline{u} > \overline{u} \right\rbrace} \\
	&\quad \leq h(x, \underline{u}) \, \underline{u} \, \chi_{\left\lbrace \underline{u} > \overline{u} \right\rbrace} \in L^{1}(\Omega),
	\end{aligned}
	\end{equation*}
	where the last integrability follows from Definition~\ref{definition2}. Therefore, combining this with the reasoning used in the proof of Theorem~\ref{Theorem1}, we conclude that
	\begin{equation*}
	\int_{\left\lbrace \underline{u} > \overline{u} \right\rbrace} \left( \frac{h(x, \underline{u})}{\underline{u}^{p^{-}-1}} - \frac{h(x, \overline{u})}{\overline{u}^{p^{-}-1}} \right)(\underline{u}^{p^{-}} - \overline{u}^{p^{-}})\, dx \leq 0.
	\end{equation*}
	This implies that \( \underline{u} \leq \overline{u} \) almost everywhere in \( \Omega \), which completes the proof.
\end{proof}

\noindent
Building on the preceding results, the uniqueness of solutions to~\eqref{GP} follows directly, as demonstrated below:
\begin{proof}[\textbf{Proof of Corollary~\ref{corollary}}]
	Suppose that \( u \) and \( v \) are two weak solutions of problem~\eqref{GP} in \( W^{s, G}_{\mathrm{loc}}(\Omega) \). By treating \( u \) as a sub-solution and \( v \) as a super-solution of~\eqref{GP}, the weak comparison principle (Theorem~\ref{Theorem1}; see also Theorem~\ref{Theorem2} for the case where \( u, v \in W^{s, G}_{0}(\Omega) \)) implies that \( u \leq v \) almost everywhere in \( \Omega \). Reversing the roles of \( u \) and \( v \), we similarly obtain \( v \leq u \) a.e. in \( \Omega \). Hence, it follows that \( u = v \) a.e. in \( \Omega \).
\end{proof}

\begin{proof}[\textbf{Proof of Corollary~\ref{corollary2}}]
	Without loss of generality, and due to the invariance of the problem under rotations and translations, we may assume that \( \Omega \) is symmetric with respect to the \( x_1 \)-axis. Accordingly, the symmetry of the data implies that \( f(x_1, x') = f(-x_1, x') \) and \( k(x_1, x') = k(-x_1, x') \) for all \( x' \in \mathbb{R}^{N-1} \). Define the function \( v(x_1, x') := u(-x_1, x') \). It is straightforward to verify that \( v \) is a weak solution of problem~\eqref{GP}. Therefore, by the uniqueness result established in Corollary~\ref{corollary}, we deduce that
	\[
	u(x_1, x') = v(x_1, x') = u(-x_1, x') \quad \text{for all } (x_1, x') \in \Omega,
	\]
	which establishes the desired symmetry.
\end{proof}
\section*{Statements and Declarations}
\subsection*{Ethics approval and consent to participate}
Not applicable.
\subsection*{Funding}
Not applicable
\subsection*{Availability of data and materials}
Not applicable
\subsection*{Conflict of interest}
The authors declare that there is no conflict of interest.

\bibliographystyle{plain}
\bibliography{BIB}
\end{document}